\documentclass[11pt,reqno]{amsart}
\usepackage[a4paper,margin=1in]{geometry}

\usepackage[english]{babel}
\usepackage[dvips]{epsfig}
\usepackage{xcolor}
\usepackage{graphicx}
\usepackage{pict2e}
\usepackage{enumitem}

\usepackage{hyperref}
\def\Itemautorefname~#1\null{(#1)\null}

\usepackage{amsmath, amsthm, amssymb}
\usepackage{amsfonts}
\usepackage{mathrsfs}
\usepackage{mathtools}          
\usepackage{esint}              
\usepackage[thinc]{esdiff}

\theoremstyle{plain}
\newtheorem{theorem}{Theorem}[section]
\newtheorem{lemma}[theorem]{Lemma}
\newtheorem{proposition}[theorem]{Proposition}
\newtheorem{corollary}[theorem]{Corollary}

\theoremstyle{definition}
\newtheorem{definition}[theorem]{Definition}
\newtheorem{remark}[theorem]{Remark}
\newtheorem*{remark*}{Remark}

\newtheorem*{assumption*}{Assumption}

\numberwithin{equation}{section}

\usepackage[
backend=biber,
style=alphabetic,
sorting=nyt,
giveninits=true,
]{biblatex}
\usepackage[style=american]{csquotes}
\renewbibmacro{in:}{%
  \ifboolexpr{%
     test {\ifentrytype{article}}%
  }{}{\printtext{\bibstring{in}\intitlepunct}}%
}
\DeclareFieldFormat[article]{pages}{#1}
\DeclareFieldFormat[incollection]{pages}{#1}
\DeclareFieldFormat[article]{title}{\mkbibemph{#1}}
\DeclareFieldFormat[article]{journaltitle}{#1}

\renewbibmacro*{note+pages}{%
  \iffieldundef{note}
    {}
    {\setunit{\addcomma\space}}%
  \printfield{note}%
  \setunit{\bibpagespunct}%
  \printfield{pages}%
  \newunit}

\usepackage{dsfont}
\usepackage{bbm}

\newcommand{\R}{\mathbb{R}}

\newcommand{\cH}{\mathcal{H}}

\DeclareMathOperator{\diam}{diam}
\DeclareMathOperator{\dist}{dist}
\DeclareMathOperator{\Div}{div}

\DeclareMathOperator{\osc}{osc}

\DeclareMathOperator{\supp}{supp}

\newcommand{\xmathpalette}[2]{%
  \mathchoice
    {#1\displaystyle\textfont{#2}}
    {#1\textstyle\textfont{#2}}
    {#1\scriptstyle\scriptfont{#2}}
    {#1\scriptscriptstyle\scriptscriptfont{#2}}
}
\makeatletter
\newcommand{\mres@thickness}[1]{\dimexpr1.5\fontdimen8 #13\relax}
\newcommand{\mres}{\mspace{3mu}{\xmathpalette\mres@\relax}\mspace{3mu}}
\newcommand{\mres@}[3]{%
  \begingroup
  \setlength\unitlength{%
    \dimexpr\fontcharht#21`A-0.5\mres@thickness{#2}
  }%
  \raisebox{0.5\dimexpr\mres@thickness{#2}}{%
    \begin{picture}(1,1)
      \roundcap\roundjoin
      \linethickness{\mres@thickness{#2}}
      \polyline(0,1)(0,0)(1,0)
    \end{picture}%
   }%
  \endgroup
}
\makeatother

\begin{document}








\allowdisplaybreaks

\title[Regularity of free boundaries of Alt-Caffarelli type
problems]{Regularity of Lipschitz free boundaries for weak solutions of
  Alt-Caffarelli type problems}

\author{Joan Domingo-Pasarin}

\author{Xavier Ros-Oton}

\address{Departament de Matem\`atiques i Inform\`atica, Universitat de Barcelona, Gran Via de les Corts Catalanes 585, 08007 Barcelona, Spain}
\email{jdomingopasarin@ub.edu}

\address{ICREA, Pg. Llu\'is Companys 23, 08010 Barcelona, Spain \& Universitat de Barcelona, Departament de Matem\`atiques i Inform\`atica, Gran Via de les Corts Catalanes 585, 08007 Barcelona, Spain \& Centre de Recerca Matem\`atica, Barcelona, Spain}
\email{xros@icrea.cat}




\begin{abstract}
  Motivated by the Serrin problem, we study weak solutions of the generalised
  Alt-Caffarelli problem \(-\Delta u = f\) in \(\Omega\), \(u = 0\) on
  \(\partial \Omega\), \(\partial_{\nu}u = Q\) on \(\partial \Omega\). Our main result establishes that if
  \(\Omega\) is Lipschitz, then it is actually \(C^{\infty}\) (provided that \(f\) and
  \(Q\) are smooth). This was known before only for viscosity solutions. As a
  corollary, we obtain an alternative solution of Serrin's problem in the case of
  Lipschitz domains. We also discuss the characterisation of the regularity of
  Lipschitz domains in terms of their Poisson kernel.
\end{abstract}

\allowdisplaybreaks

\maketitle
\section{Introduction}

Free boundary problems have been a major line of research in the field of PDE in the
last decades; see for instance \cite{MR2145284, MR2962060, MR4807210}. Among them,
Serrin's overdetermined problem has seen important developments since its
introduction in \cite{MR333220}; we refer the reader to the surveys
\cite{schaefer2001non, MR3802818} for a comprehensive overview of the contributions
to this problem. Given a bounded domain \(\Omega \subset \R^n\), Serrin's problem consists in
proving that if \(\Omega\) admits a solution of
\begin{equation}
  \label{eq:serrin-prob}
  \left\{
  \begin{aligned}
    -\Delta u &= 1 \quad&&\text{in } \Omega \\
    u &= 0 \quad&&\text{on } \partial \Omega \\
    \partial_{\nu}u &= c \quad&&\text{on } \partial \Omega,
  \end{aligned}
  \right.
\end{equation}
then \(\Omega\) is a ball \(B_R\) of radius \(R = R(n,c)\). Without the Neumann condition,
\eqref{eq:serrin-prob} already has a unique solution. Therefore, the Neumann
condition makes the problem overdetermined and turns it into a free boundary problem.
Serrin's problem can thus be understood as showing that the overdetermined nature of
\eqref{eq:serrin-prob} forces \(\Omega\) to be a ball.

In his celebrated paper \cite{MR333220}, Serrin solved this problem for \(C^2\)
bounded domains under the assumption \(u \in C^2(\overline{\Omega})\). Observe that whenever
\(\Omega\) has finite \((n-1)\)-dimensional Hausdorff measure and
\(u \in H^1_0(\Omega)\), \eqref{eq:serrin-prob} has a natural weak formulation:
\begin{equation}
  \label{eq:weak-serrin-prob}
  -\int_\Omega \nabla u\cdot\nabla\eta dx = c\int_{\partial \Omega} \eta d\cH^{n-1} - \int_\Omega \eta dx \qquad \forall\eta \in C^{\infty}_{\mathrm{c}}(\R^n).
\end{equation}
Hence, after Serrin's result it was natural to ask whether one could assume lower
regularity of \(\Omega\) and \(u\) and still obtain the same result. This proved to be a
challenging question which did not see any progress initially. In 1992, Vogel
\cite{MR1200301} achieved a breakthrough when he was able to show that \(C^1\)
regularity combined with some additional hypothesis on \(u\) were enough to prove
Serrin's theorem. Six years later, in 1998, Parjapat proved in \cite{MR1487978} that
Serrin's theorem still holds for domains which are \(C^2\) everywhere except at a
potential corner. This raised the question of whether the regularity of \(\Omega\) could
be lowered even further to Lipschitz regularity. Such question was left unanswered
for 27 years until the very recent paper by Figalli and Zhang
\cite{figalli2024serrin}, where it has been solved. In it, the authors establish
Serrin's theorem for a wider class of sets of finite perimeter which includes
Lipschitz domains.

Of particular interest to us is Remark 1.5 in \cite{figalli2024serrin}, where Figalli
and Zhang highlight the connection between Serrin's problem and the Alt-Caffarelli
problem. More precisely, they point out the fact that one could think of using the
regularity theory of the latter to solve the former. Unfortunately, at the time this
regularity theory was mainly developed for viscosity-type solutions or minimizers of
the problem's energy functional. Thus, Figalli and Zhang conclude that it is unclear
how one would apply said regularity theory to the weak formulation
\eqref{eq:weak-serrin-prob}. Consequently, \cite{figalli2024serrin} developed a
different approach based on geometric measure theory techniques.

Motivated by the remark of Figalli and Zhang, in this paper we study weak solutions
of the Alt-Caffarelli problem whose free boundary is Lipschitz. Introduced in the
seminal paper \cite{MR618549}, the classical Alt-Caffarelli problem (or one-phase
Bernoulli problem) consists in studying for which bounded domains \(\Omega \subset \R^n\) (with
\(0 \in \partial \Omega\)) there exists a non-negative function \(u:B_1 \to \R\) solving
\begin{equation}
  \label{eq:alt-caff-prob}
  \left\{
  \begin{aligned}
    \Delta u &= 0 \quad&&\text{in } \Omega \cap B_1 \\
    u &= 0 \quad&&\text{on } \partial \Omega \cap B_1 \\
    \partial_{\nu}u &= 1 \quad&&\text{on } \partial \Omega \cap B_1.
  \end{aligned}
  \right.
\end{equation}
Here, \(\nu\) denotes the inward normal vector to \(\partial \Omega \cap B_1\). As in Serrin's
problem, the overdetermined nature of \eqref{eq:alt-caff-prob} leads one to expect
some restriction on the choice of \(\Omega\). In particular, one may think of using the
additional boundary condition \(\partial_{\nu}u = 1\) to prove higher regularity of the free
boundary \(\partial \Omega \cap B_1\). This topic was already studied in \cite{MR618549}
extensively. However, even though most results in \cite{MR618549} were stated for a
weak notion of solution, the literature for this problem has been mostly focused on
viscosity solutions and energy minimizers; see for instance \cite{MR990856,
  MR1029856, MR973745, MR2082392, MR2572253, MR2813524, MR3385632, MR3896024,
  MR4101738, MR4661533, ferreri2024boundarybranchonephase, fernándezreal2023generic}.
In particular, for such solutions it is known that Lipschitz free boundaries are
smooth (we refer to \cite{MR2145284} for more details). However, as Figalli and Zhang
observed, there is no analogous result for weak solutions.

\begin{remark}
  The recent paper \cite{MR4850026} provides a deep understanding of so-called
  stationary solutions of \eqref{eq:alt-caff-prob}. This class is different from the
  weak solutions that we study here. Actually, as explained in more detail below, the
  key point of our work consists in showing that, if the free boundary is Lipschitz,
  then weak solutions are stationary (and even viscosity) solutions.
\end{remark}

\subsection{Main results}

In this paper we prove the smoothness of Lipschitz free boundaries for weak solutions
of the one-phase Bernoulli problem with right hand side
\begin{equation}
  \label{eq:inhom-alt-caff-prob}
  \left\{
  \begin{aligned}
    -\Delta u &= f \quad&&\text{in } \Omega \cap B_1 \\
    u &= 0 \quad&&\text{on } \partial \Omega \cap B_1 \\
    \partial_{\nu}u &= Q \quad&&\text{on } \partial \Omega \cap B_1.
  \end{aligned}
  \right.
\end{equation}
We assume \(\Omega \subset \R^n\) to be a bounded Lipschitz domain with
\(0 \in \partial \Omega\) and both \(f\) and \(Q\) to be smooth functions such that
\(f \ge 0\) and \(Q > 0\). Our main theorem is as follows:
\begin{theorem}
  \label{thm:1}
  Let \(\Omega \subset \R^n\) be a bounded Lipschitz domain such that
  \(0 \in \partial \Omega\) and assume \(u\) solves \eqref{eq:inhom-alt-caff-prob} in the following
  weak sense: \(u \in H^1_{\mathrm{loc}}(B_1)\), \(u \geq 0\), \(u = 0\) on
  \(B_1 \setminus \Omega\) and
  \[
    -\int_{B_1} \nabla u\cdot\nabla\eta dx = \int_{\partial \Omega \cap B_1} Q\eta d\cH^{n-1} - \int_{\Omega \cap B_1} f\eta dx,
    \qquad \forall\eta \in C^{\infty}_{\mathrm{c}}(B_1).
  \]
  If \(f,Q \in C^{\infty}(B_1)\), \(f \geq 0\) and \(Q > 0\), then
  \(\partial \Omega \cap B_1\) is locally a smooth surface in \(B_1\).
\end{theorem}
It is important to highlight the fact that the Lipschitz regularity of \(\Omega\) in
Theorem \ref{thm:1} is optimal. Namely, it is not possible to prove smoothness of the
free boundary (for weak solutions) if we assume lower regularity of \(\Omega\). Indeed, in
\cite{MR1109479, MR1891198} Lewis and Vogel showed that it is possible to construct a
domain \(\Omega\) whose boundary is \(C^{1-\varepsilon}\) for every
\(\varepsilon \in (0,1)\), and for which there exists a function \(u\) solving
\eqref{eq:inhom-alt-caff-prob} in the weak sense of Theorem \ref{thm:1} with
\(f \equiv 0\) and \(Q \equiv 1\).

This, in particular, shows that weak solutions could exhibit a more delicate behavior
at the free boundary \(\partial \Omega \cap B_1\) than viscosity solutions or energy minimizers.
Actually, a blow-up procedure at the free boundary does not preserve the weak notion
of solution in general. In fact, this only happens for those solutions for which the
set \(B_1 \setminus \Omega\) has positive density at the free boundary. Moreover, even with this
assumption, the improvement of flatness from \cite{MR3667700} is still needed to show
that the blow-up limit is a weak solution. Compare this with the setting of viscosity
solutions and energy minimizers, which pass to the limit without the need for any
major result or assumption. In a similar fashion, we also need \cite{MR3667700} to
prove that weak solutions have zero first inner variation, which is then used in the
proof of the 1-homogeneity of blow-ups.

As we see, the improvement of flatness from \cite{MR3667700} is of utmost importance
when dealing with weak solutions of \eqref{eq:inhom-alt-caff-prob} and proving
Theorem \ref{thm:1}. Unfortunately, even though \cite{MR3667700} deals with the even
more general setting of the \(p(x)\)-Laplacian, we cannot use directly their results
due to the fact that their weak notion of solution is more restrictive than ours.
Showing that our definition of weak solution fulfills the hypothesis in
\cite{MR3667700} is far from trivial and constitutes one of the main steps of the
proof of Theorem \ref{thm:1}.

As an immediate consequence of Theorem \ref{thm:1}, we are able to give an
alternative proof of \cite[Theorem 1.3]{figalli2024serrin} for the case of Lipschitz
domains.
\begin{theorem}
  \label{thm:2}
  Let \(\Omega \subset \R^n\) be a bounded Lipschitz domain. Then \(\Omega\) admits a solution
  \(u \in H^1_0(\Omega)\) to \eqref{eq:weak-serrin-prob} if and only if \(\Omega\) is a ball.
\end{theorem}

\subsection{Poisson kernel}

The regularity theory for the one-phase Bernoulli problem
\eqref{eq:inhom-alt-caff-prob} can also be used to characterize the regularity of
Lipschitz domains in terms of their Poisson kernel.

Given a bounded Lipschitz domain \(\Omega \subset \R^n\), consider the Green function
\(G_x\) for \(\Omega\) with pole at \(x\). Then, \(G_x\) is a positive harmonic function
in \(\Omega \setminus \{x\}\) which is continuous up to the boundary of \(\Omega\) and has zero
boundary value on \(\partial \Omega\). Moreover, by Dahlberg's theorem \cite[Theorem
13.10]{tolsa2025prats}, the Poisson kernel \(P_x = \partial_{\nu}G_x\), where \(\nu\) is the
inward normal vector to \(\Omega\), exists \(\cH^{n-1}\)-a.e. on \(\partial \Omega\) and is positive.

The regularity of \(P_x\) is deeply linked to the regularity of the domain
\(\Omega\). On the one hand, if \(\Omega\) is smooth, then
\(G_x \in C^{\infty}(\overline{\Omega} \setminus \{x\})\) and \(P_x\) is smooth for any
\(x \in \Omega\). Thus, the smoothness of \(\Omega\) implies the smoothness of
\(P_x\). On the other hand, one may wonder if the converse is also true. This has
been widely studied in the last decades (see e.g. \cite{MR1096396, MR3540451,
  MR3641880, MR4169053}) and is now well understood. In particular, the following
result has been known for a long time.
\begin{theorem}
  \label{thm:3}
  Let \(\Omega \subset \R^n\) be a bounded Lipschitz domain and \(x \in \Omega\). Assume that the
  Poisson kernel \(P_x\) for \(\Omega\) with pole at \(x\) equals
  \(Q|_{\partial\Omega}\) for some positive function \(Q \in C^{\infty}(\R^n)\). Then
  \(\partial \Omega\) is \(C^{\infty}\).
\end{theorem}
Although this result is usually attributed to \cite{MR990856} (see, e.g.
\cite{MR1096396}), the results in \cite{MR990856} are for viscosity solutions and it
is not obvious how to apply them to deduce Theorem \ref{thm:3}. Our Theorem
\ref{thm:1} clarifies this issue, as it applies directly to weak solutions.

Indeed, in this context, one can show that the Green function \(G_x\) is a weak
solution of \eqref{eq:inhom-alt-caff-prob} (with \(f \equiv 0\)). For this, by \cite[Lemma
7.6]{tolsa2025prats}
\[
  -\int_\Omega \nabla G_x\cdot\nabla\eta = -\eta(x) + \int \eta d\omega^x \qquad \forall\eta \in C^{\infty}_{\mathrm{c}}(\R^n),
\]
where \(\omega^x\) is the harmonic measure for \(\Omega\) with pole at \(x\). Since
\(d\omega^x = P_xd\cH^{n-1}\) by Dahlberg's theorem, and localizing in a ball \(B_r\) such
that \(x \not\in B_r\) (we assume \(0 \in \partial \Omega\)), we have
\begin{equation}
  \label{eq:green-fun-weak-sol}
  -\int_{B_r} \nabla G_x\cdot\nabla\eta = \int_{\partial \Omega \cap B_r} P_x\eta d\cH^{n-1} \qquad \forall\eta \in C^{\infty}_{\mathrm{c}}(B_r).
\end{equation}
That is, \(G_x\) (extended by zero in \(B_r \setminus \Omega\)) solves
\eqref{eq:inhom-alt-caff-prob} in the weak sense with \(Q = P_x\). This allows us to
apply Theorem \ref{thm:1} to \(G_x\) to recover Theorem \ref{thm:3}.
\begin{remark}
  Theorems \ref{thm:1} and \ref{thm:3} still hold if we assume less regularity of
  \(f\) and \(Q\). Precisely, the results from \cite{carducci2025regularity} imply
  that if \(f \in C^{k-1,\alpha}(B_1)\) and \(Q \in C^{k,\alpha}(B_1)\) with
  \(k \geq 2\) and \(\alpha > 0\), then \(\partial\{u > 0\} \cap B_1\) and
  \(\partial\Omega\) are \(C^{k+1,\alpha}\) respectively.
\end{remark}
\begin{remark}
  Even though Theorem \ref{thm:3} only needs \(P_x\) to be smooth for one point
  \(x \in \Omega\), once we know \(\partial \Omega\) is smooth it follows that
  \(P_x\) will be smooth for all \(x \in \Omega\).
\end{remark}

Lastly, let us expand a bit more on the domains constructed in \cite{MR1109479,
  MR1891198} mentioned previously. In \cite{MR1109479}, Lewis and Vogel were able to
construct a non-trivial domain \(\Omega\) with boundary homeomorphic to the unit sphere,
and such that the Poisson kernel for \(\Omega\) at \(0 \in \Omega\) is given by a constant
\(a\). Moreover, they showed that both the homeomorphism and its inverse can be taken
to be \(C^{1-\varepsilon}\) for every \(\varepsilon \in (0,1)\). Thus, even though the Green function
\(G_0\) for this domain is a weak solution of \eqref{eq:inhom-alt-caff-prob} near
\(\partial\Omega\) with \(Q \equiv a \in C^{\infty}\), the boundary of \(\Omega\) is not smooth. In other words,
\cite{MR1109479, MR1891198} show that the Lipschitz regularity of \(\Omega\) in Theorem
\ref{thm:3} is optimal in terms of proving smoothness of \(\partial\Omega\).

\subsection{Strategy of the proof}

As we mentioned previously, an important ingredient we use in the proof of Theorem
\ref{thm:1} is the improvement of flatness from \cite{MR3667700}. To be able to apply
it, we need to generalize some results from \cite{MR618549} to show that our
definition of weak solution fulfills the more restrictive hypothesis in
\cite{MR3667700}.

Once we have an improvement of flatness result, we obtain as a consequence the fact
that blow-up limits preserve the notion of weak solution. Importantly, this only
holds if we also assume that the set \(B_1 \setminus \Omega\) has positive density at free
boundary points. This assumption, which is always satisfied by Lipschitz free
boundaries, is needed to rule out degenerate limits such as \(|x\cdot e_n|\) which, a
priori, can happen without the density condition. Also as a consequence of the
improvement of flatness we obtain that the first inner variation of weak solutions
vanishes. We then use this result to prove by means of a Weiss-type monotonicity
formula the 1-homogeneity of blow-up limits.

At this point we have all the main ingredients of the proof of Theorem \ref{thm:1}.
The last step of the proof can be done in two different ways, both using all the
results mentioned above. The first one consists in classifying the 1-homogenous weak
solutions obtained as blow-up limits at free boundary points. To do so, we follow the
ideas of \cite{MR3353802} to prove that said limits are flat. From there, the
improvement of flatness allows us to conclude the proof. Alternatively, we can use
the already known regularity theory for viscosity solutions proven by De Silva in
\cite{MR2813524}. Assuming once more the density condition of the set
\(B_1 \setminus \Omega\), we are able to show that weak solutions of
\eqref{eq:inhom-alt-caff-prob} are also viscosity solutions. Thus, by
\cite{MR2813524} we obtain the smoothness of the free boundary.

\subsection{Organization of the paper}

The paper is organised as follows. In Section
\ref{sec:weak-solutions-basic-properties} we define the notion of weak solution for
\eqref{eq:inhom-alt-caff-prob} and prove some its basic properties. In Section
\ref{sec:impr-flatn-weak} we show that our definition of weak solution satisfies the
required hypothesis to apply the improvement of flatness result from
\cite{MR3667700}. In Section \ref{sec:blow-ups-weak} we study the blow-up sequences
of weak solutions at the free boundary. In Section \ref{sec:homogeneity-blow-ups} we
prove the 1-homogeneity of blow-up limits. In Section \ref{sec:lip-implies-smooth} we
show the smoothness of Lipschitz free boundaries with smooth data \(f,Q\). Lastly, in
Section \ref{sec:main-results} we prove Theorem \ref{thm:1} and its consequences
Theorems \ref{thm:2} and \ref{thm:3}.

\subsection{Acknowledgements}

The authors were supported by the European Research Council under the Grant Agreement
No.~101123223 (SSNSD), and by AEI project PID2024-156429NB-I00 (Spain). Moreover,
X.~Ros-Oton was supported by the AEI grant RED2024-153842-T (Spain), by the AEI–DFG
project PCI2024-155066-2 (Spain–Germany), and by the Spanish State Research Agency
through the Mar\'ia de Maeztu Program for Centers and Units of Excellence in R\&D
(CEX2020-001084-M).

\section{Weak solutions and basic properties}
\label{sec:weak-solutions-basic-properties}

To study \eqref{eq:inhom-alt-caff-prob} we consider the following equivalent
formulation:
\begin{equation}
  \label{eq:inhom-alt-caff-free-bound-form}
  \left\{
  \begin{aligned}
    -\Delta u &= f \quad&&\text{in } \{u > 0\} \cap B_1 \\
    u &= 0 \quad&&\text{on } \partial\{u > 0\} \cap B_1 \\
    \partial_{\nu}u &= Q \quad&&\text{on } \partial\{u > 0\} \cap B_1.
  \end{aligned}
  \right.
\end{equation}
Here, \(u:B_1 \to \R\) is non-negative and \(f\) and \(Q\) are given smooth functions
in \(B_1\). As in \cite{MR618549}, throughout the paper we assume
\[
  0 < Q_{\mathrm{min}} \leq Q \leq Q_{\mathrm{max}} < \infty,
\]
and we consider the case \(f \geq 0\). Observe that \(f \equiv 0\) corresponds to the
classical homogeneous case studied by Alt and Caffarelli in \cite{MR618549}.
\begin{definition}
  \label{defn:1}
  A function \(u\) is a weak solution of \eqref{eq:inhom-alt-caff-free-bound-form} in
  \(B_1\) if it satisfies the following conditions:
  \begin{enumerate}
  \item\label{item:1} \(u \in C(B_1)\), \(u \ge 0\) and \(-\Delta u = f\) in the open set
    \(\{u > 0\}\).
  \item\label{item:2} For any open set \(D \subset\subset B_1\) there exist constants
    \(0 < c_D \le C_D < \infty\) such that for any ball \(B_r(x) \subset D\) with
    \(x \in \partial\{u > 0\}\) we have
    \[
      c_D \le \frac{1}{r}\fint_{\partial B_r(x)} u d\cH^{n-1} \le C_D.
    \]
  \item\label{item:3} \(u\) solves the equation
    \[
      \Delta u = Q\cH^{n-1}\mres\partial_{\mathrm{red}}\{u > 0\} - f\chi_{\{u > 0\}}
    \]
    in \(B_1\) in the distributional sense, that is,
    \[
      -\int_{B_1}\nabla u\cdot\nabla\eta dx =
      \int_{\partial_{\mathrm{red}}\{u > 0\} \cap B_1} Q\eta d\cH^{n-1} - \int_{B_1} f\chi_{\{u > 0\}}\eta dx
      \quad\text{for any } \eta \in C^{\infty}_{\mathrm{c}}(B_1).
    \]
  \end{enumerate}
\end{definition}
\begin{remark}
  If \(u:B_1 \to \R\) fulfils condition \autoref{item:1} of Definition \ref{defn:1},
  then \(u \in H^1_{\mathrm{loc}}(B_1)\). This follows from a similar argument to
  \cite[Remark 4.2]{MR618549}, adjusting for the fact that \(-\Delta u = f\).
\end{remark}
\begin{remark}
  \label{rem:1}
  The second condition in Definition \ref{defn:1} is a regularity and non-degeneracy
  condition. Indeed, let \(u:B_1 \to \R\) be a function satisfying conditions
  \autoref{item:1} and \autoref{item:2}. Consider an open set
  \(D \subset\subset B_1\) and let \(B_r(x) \subset D\) be any ball with
  \(x \in \partial\{u > 0\}\). Integrating in spherical coordinates in \autoref{item:2} yields
  \[
    c \leq \frac{1}{r}\fint_{B_r(x)} u dx \leq C
  \]
  for some constants \(c\), \(C\) depending on \(D\). In particular,
  \[
    c \leq \frac{1}{r}\fint_{B_r(x)} u dx \leq \frac{1}{r}\sup_{B_r(x)} u,
  \]
  and we see that the lower bound in \autoref{item:2} corresponds to a non-degeneracy
  condition for \(u\). Moreover, as we will see in Proposition \ref{prop:1}, the
  upper bound in \autoref{item:2} implies that \(u\) is locally Lipschitz in \(B_1\).
  Thus, we can write
  \[
    \tilde{c} \leq \frac{1}{r}\sup_{B_r(x)} u \leq \tilde{C},
  \]
  where \(\tilde{c}\) and \(\tilde{C}\) are positive constants depending on
  \(D\).
\end{remark}
\begin{remark}
  \label{rem:2}
  By Remark \ref{rem:1}, if \(u\) satisfies the first two conditions of Definition
  \ref{defn:1}, then \(u\) satisfies (1) and (2) of \cite[Definition 2.2]{MR3667700}.
  Thus, \cite[Theorem 2.1, Remark 2.3]{MR3667700} imply that the set \(\{u > 0\}\)
  has locally finite perimeter in \(\Omega\). Hence, the reduced boundary
  \(\partial_{\mathrm{red}}\{u > 0\}\) in \autoref{item:3} is well defined.
\end{remark}
\begin{remark}
  \label{rem:3}
  As one would expect, any weak solution with smooth free boundary is a
  classical solution. Indeed, if \(\partial\{u > 0\}\) is regular enough (say
  \(C^{1,\alpha}\)), then
  \(\partial\{u > 0\} = \partial_{\mathrm{red}}\{u > 0\}\) and by Schauder estimates
  \(u \in C^{1,\alpha}(\overline{\{u > 0\}} \cap B_1)\). Thus, given
  \(\eta \in C^{\infty}_{\mathrm{c}}(B_1)\), an integration by parts yields
  \begin{align*}
    \int_{B_1} \nabla u \cdot \nabla\eta dx &= -\int_{\partial\{u > 0\} \cap B_1} \eta\nabla u\cdot\nu d\cH^{n-1}
                        - \int_{\partial\{u > 0\} \cap \partial B_1} \eta\nabla u\cdot\nu d\cH^{n-1}
                        - \int_{\{u > 0\}} \eta\Delta u dx \\
                      &= -\int_{\partial\{u > 0\} \cap B_1} \eta\partial_{\nu}u d\cH^{n-1} + \int_{B_1} f\chi_{\{u > 0\}}\eta dx,
  \end{align*}
  where \(\nu\) is the inward normal vector of \(\{u > 0\}\). Consequently, condition
  \autoref{item:3} implies that
  \[
    \int_{\partial\{u > 0\} \cap B_1} (Q - \partial_{\nu}u)\eta d\cH^{n-1} = 0 \quad\text{for all } \eta \in C^{\infty}_{\mathrm{c}}(B_1)
  \]
  and therefore \(|\nabla u| = \partial_{\nu}u = Q\) on \(\partial\{u > 0\} \cap B_1\).
\end{remark}
\begin{proposition}
  \label{prop:1}
  Let \(u:B_1 \to \R\) be a function satisfying condition \autoref{item:1} and the
  upper bound of condition \autoref{item:2} of Definition \ref{defn:1}. Then \(u\) is
  locally Lipschitz in \(B_1\).
\end{proposition}
\begin{proof}
  We will prove that for each small \(\delta > 0\) there is a constant \(C_{\delta}\) such that
  \(\|\nabla u\|_{L^{\infty}(B_{1-\delta})} \leq C_{\delta}\). Since \(\nabla u = 0\) in the interior of
  \(\{u = 0\} \cap B_{1-\delta}\), it is enough to estimate the gradient of \(u\) in
  \(\{u > 0\} \cap B_{1-\delta}\). Let \(x_0 \in \{u > 0\} \cap B_{1-\delta}\). We follow the proof of
  \cite[Lemma 3.5]{MR4807210} and consider two cases:
  \begin{itemize}
  \item If \(\dist(x_0,\partial\{u > 0\}) \geq \delta/6\), then \(-\Delta u = f\) in the ball
    \(B_{\delta/6}(x_0)\) and so, by classical Schauder estimates we have
    \begin{align*}
      |\nabla u(x_0)| &\leq C_n\left(\frac{1}{\delta}\|u\|_{L^{\infty}(B_{\delta/6}(x_0))} + \delta\|f\|_{L^{\infty}(B_{\delta/6}(x_0))}\right) \\
                 &\leq C_n\left(\frac{1}{\delta}\|u\|_{L^{\infty}(B_{1-\frac{\delta}{2}})} + \delta\|f\|_{L^{\infty}(B_{1-\frac{\delta}{2}})}\right).
    \end{align*}
    Here we have used that
    \(B_{\delta/6}(x_0) \subset B_{1-\frac{5}{6}\delta} \subset B_{1-\frac{\delta}{2}}\) since \(x_0 \in B_{1-\delta}\).
  \item If \(\dist(x_0,\partial\{u > 0\}) < \delta/6\), then we let
    \(y_0 \in \partial\{u > 0\}\) be the point that realizes the distance to the free boundary
    and set
    \[
      r \coloneqq \dist(x_0,\partial\{u > 0\}) = |x_0 - y_0|.
    \]
    Since \(-\Delta u = f\) in \(B_r(x_0)\), we can use interior estimates once more to
    obtain
    \[
      \|\nabla u\|_{L^{\infty}(B_{r/2}(x_0))} \leq C_n\left(\frac{1}{r^{n+1}}\|u\|_{L^1(B_r(x_0))} + r\|f\|_{L^{\infty}(B_r(x_0))}\right).
    \]
    Hence,
    \begin{align*}
      |\nabla u(x_0)| &\leq C_n\left(\frac{1}{r^{n+1}}\|u\|_{L^1(B_r(x_0))} + r\|f\|_{L^{\infty}(B_r(x_0))}\right) \\
                 &= C_n\left(\frac{1}{r}\fint_{B_r(x_0)} u dx + r\|f\|_{L^{\infty}(B_r(x_0))}\right) \\
                 &\leq C_n\left(\frac{2^n}{r}\fint_{B_{2r}(y_0)} u dx + \frac{\delta}{6}\|f\|_{L^{\infty}(B_{1-\frac{5}{6}\delta})}\right) \\
                 &\leq C_n\left(2^{n+1}C_{B_{1-\frac{\delta}{2}}} + \frac{\delta}{6}\|f\|_{L^{\infty}(B_{1-\frac{\delta}{2}})}\right).
    \end{align*}
    For the last inequality we have used the fact that for any open set
    \(D \subset\subset B_1\), any \(x \in \partial\{u > 0\} \cap D\) and any ball \(B_r(x) \subset D\) we have
    \[
      \frac{1}{r}\fint_{B_r(x)} u dx \leq C_D
    \]
    for some constant \(C_D\) depending on \(D\). This is an immediate consequence of
    integrating in spherical coordinates in \autoref{item:2}.
  \end{itemize}
\end{proof}
\begin{remark}
  \label{rem:4}
  In the classical Alt-Caffarelli problem \(f \equiv 0\), \(Q \equiv 1\), weak solutions are
  stationary solutions (see Definition \ref{defn:2}) by \cite[Theorem
  5.1]{MR1620644}. Also, minimizers of \(\mathcal{F}\) are both viscosity and weak solutions by
  \cite{MR4807210, MR618549} respectively. In other words, we have the following
  relations between the different notions of solutions:
  \[
    \begin{array}{ccccc}
      \text{Minimizer} & \implies & \text{Weak sol.} & \implies & \text{Stationary sol.} \\
      \big\Downarrow &&&& \\
      \text{Viscosity sol.} &&&&
    \end{array}
  \]
  Note that weak solutions are a proper subclass of stationary solutions. Indeed, it
  is easy to check that the function \(u(x) = |x_n|\) is a stationary solution in
  \(\R^n\) for \(f \equiv 0\) and \(Q \equiv 1\). However, \(u\) is not a weak solution since
  \(\partial_{\mathrm{red}}\{u > 0\} = \emptyset\) and condition \autoref{item:3} would imply that
  \(u\) is harmonic in \(\R^n\). One of the main results of this paper is to show
  that for the more general case \(f \geq 0\), weak solutions are stationary solutions
  (see Proposition \ref{prop:6}) and, assuming an additional condition on the free
  boundary, viscosity solutions (see Proposition \ref{prop:10}).
\end{remark}
\begin{lemma}
  \label{lem:1}
  Assume that \(u:B_1 \to \R\) is a function satisfying conditions \autoref{item:1} and
  \autoref{item:2} of Definition \ref{defn:1}. Let \(B_{r_k}(x_k) \subset B_1\) be a
  sequence of balls with \(r_k \to 0\), \(x_k \to x_0 \in B_1\) and \(u(x_k) = 0\).
  Consider the blow-up sequence
  \[
    u_{r_k}(x) \coloneqq \frac{1}{r_k}u(x_k + r_kx).
  \]
  Then there exists a Lipschitz continuous blow-up limit \(u_0:\R^n \to \R\)
  such that, up to subsequence,
  \begin{enumerate}
  \item\label{item:4} \(u_{r_k} \rightarrow u_0\) in \(C^{0,\alpha}_{\mathrm{loc}}(\R^n)\) for every
    \(0 < \alpha < 1\),
  \item\label{item:5} \(u_{r_k} \rightarrow u_0\) in \(H^1_{\mathrm{loc}}(\R^n)\),
  \item\label{item:6} \(\chi_{\{u_{r_k} > 0\}} \rightarrow \chi_{\{u_0 > 0\}}\) in
    \(L^1_{\mathrm{loc}}(\R^n)\),
  \item\label{item:7} \(\partial\{u_{r_k} > 0\} \rightarrow \partial\{u_0 > 0\}\) locally in Hausdorff distance in
    \(\R^n\),
  \item\label{item:8} If \(x_k \in \partial\{u > 0\}\), then \(0 \in \partial\{u_0 > 0\}\),
  \item\label{item:9} \(\Delta u_0 = 0\) in \(\{u_0 > 0\}\).
  \end{enumerate}
  Moreover, \(u_0\) satisfies condition \autoref{item:2} of Definition \ref{defn:1}
  with the same constants as \(u\).
\end{lemma}
\begin{proof}
  The proof of 4.7 in \cite{MR618549} can be applied to our case with small
  modifications to account for the fact that \(-\Delta u = f\). Thus, we obtain the
  existence of a Lipschitz continuous blow-up limit \(u_0\) satisfying
  \autoref{item:4}, \autoref{item:6}, \autoref{item:7} and \autoref{item:8}.
  Moreover, we also have that \(\nabla u_{r_k} \to \nabla u_0\) weakly star in
  \(L^{\infty}_{\mathrm{loc}}(\R^n)\).

  Observe that \(u_{r_k}\) is a solution of \(-\Delta u_{r_k} = f_{r_k}\) in
  \(\{u_{r_k} > 0\}\) with \(f_{r_k}(x) = r_kf(x_k + r_kx)\) and, since \(f\) is
  smooth, \(f_{r_k} \to 0\) in \(C^{\alpha}_{\mathrm{loc}}(\R^n)\). Consider any ball
  \(B_r(x) \subset \{u_0 > 0\}\). Then, for \(k\) large enough
  \(B_r(x) \subset \{u_{r_k} > 0\}\) and
  \[
    \|u_{r_k} - u_{r_l}\|_{C^{2,\alpha}(B_{r/2}(x))} \leq
    C(\|u_{r_k} - u_{r_l}\|_{L^{\infty}(B_r(x))} + \|f_{r_k} - f_{r_l}\|_{C^{\alpha}(B_r(x))}).
  \]
  It follows that \(u_{r_k}\) is a Cauchy sequence in \(C^{2,\alpha}(K)\) for any compact
  subset \(K \subset \{u_0 > 0\}\), and therefore \(u_{r_k} \to u_0\) in
  \(C^{2,\alpha}_{\mathrm{loc}}(\{u_0 > 0\})\). In particular,
  \[
    \Delta u_0 = \lim_{r_k \to 0} \Delta u_{r_k} = -\lim_{r_k \to 0} f_{r_k} = 0
  \]
  which proves \autoref{item:9}.

  Let us prove \autoref{item:5}. First we claim that
  \(\nabla u_{r_k} \to \nabla u_0\) pointwise a.e. in \(\R^n\). To see this, observe that since
  \(u_{r_k} \to u_0\) in \(C^{2,\alpha}_{\mathrm{loc}}(\{u_0 > 0\})\), in particular
  \[
    \nabla u_{r_k} \to \nabla u_0 \quad\text{uniformly in compact subsets of } \{u_0 > 0\}.
  \]
  Thus, we only need to show that
  \begin{equation}
    \label{eq:1}
    \nabla u_{r_k} \to \nabla u_0 \quad\text{a.e. in } \{u_0 = 0\}.
  \end{equation}
  In the contact set \(\{u_0 = 0\}\) a.a. points have density 1 by the Lebesgue
  differentiation theorem. Let \(S\) be the set of such points. If \(x \in S\), then
  \begin{equation}
    \label{eq:2}
    u_0(x + y) = o(|y|).
  \end{equation}
  Indeed, otherwise we could find a sequence \(y_k \in B_{\rho_k}(x)\) with
  \(\rho_k \to 0\) such that \(u_0(y_k) > \gamma \rho_k\) for some \(\gamma > 0\). Then by the
  Lipschitz continuity of \(u_0\) we would have
  \[
    u_0 > \frac{\gamma}{2}\rho_k \quad\text{in } B_{c\gamma \rho_k}(y_k)
  \]
  for some small constant \(c > 0\) depending only on the Lipschitz constant
  of \(u_0\). This would imply that \(\{u_0 > 0\}\) has positive density at
  \(x\), contradicting the fact that \(x \in S\). From \eqref{eq:2} we have that
  for any \(\varepsilon > 0\),
  \[
    u_0 < \frac{\varepsilon}{2}\rho \quad\text{in } B_{\rho}(x) \text{ for small } \rho.
  \]
  Using that \(u_{r_k} \to u_0\) uniformly in \(B_{\rho}(x)\) we then see that
  \[
    u_{r_k} < u_0 + \frac{\varepsilon}{2}\rho < \varepsilon\rho
  \]
  in \(B_{\rho}(x)\) if \(k\) is large enough, say \(k \geq k_0(\varepsilon,\rho)\). Thus,
  \[
    \frac{u_{r_k}}{\rho} < \varepsilon \quad\text{in } B_{\rho}(x)
  \]
  and by non-degeneracy it then follows that \(u_{r_k} = 0\) in \(B_{\rho/2}(x)\).
  Consequently, \(u_0 = 0\) in \(B_{\rho/2}(x)\) and therefore \(S\) is an open set.
  Moreover, this argument also shows that \(u_{r_k} = u_0 = 0\) in any ball
  \(B_{\delta}(x) \subset S\) with \(\delta\) small enough and \(k\) sufficiently large. Hence, the
  same conclusion also holds for any compact subset of \(S\) and this completes the
  proof of \eqref{eq:1}.

  Now, since \(\nabla u_{r_k} \to \nabla u_0\) weakly star in
  \(L^{\infty}_{\mathrm{loc}}(\R^n)\), the sequence \(\nabla u_{r_k}\) is bounded in
  \(L^{\infty}_{\mathrm{loc}}(\R^n)\). Hence, by the a.e. pointwise convergence showed
  above and the dominated convergence theorem we have
  \(\nabla u_{r_k} \to \nabla u_0\) in \(L^2_{\mathrm{loc}}(\R^n)\). Also, the local uniform
  convergence of \(u_{r_k}\) implies that \(u_{r_k} \to u_0\) in
  \(L^2_{\mathrm{loc}}(\R^n)\). Thus, \(u_{r_k} \to u_0\) in
  \(H^1_{\mathrm{loc}}(\R^n)\).

  Lastly, the fact that \(u_0\) satisfies condition \autoref{item:2} of Definition
  \ref{defn:1} is an immediate consequence of the local uniform convergence
  \(u_{r_k} \to u_0\) and can be proven in the same way as 4.7 in \cite{MR618549}.
\end{proof}
\begin{remark}
  \label{rem:5}
  Using Lemma \ref{lem:1} and the same arguments of \cite[Lemma 5.3]{MR618549} we
  obtain that for any weak solution \(u\),
  \begin{enumerate}
  \item \(\cH^{n-1}(\partial\{u > 0\} \setminus \partial_{\mathrm{red}}\{u > 0\}) = 0\),
  \item \(|\{u = 0\} \cap D| > 0\) for every open set \(D \subset B_1\) containing a point of
    \(\{u = 0\}\).
  \end{enumerate}
\end{remark}

\section{Improvement of flatness for weak solutions}
\label{sec:impr-flatn-weak}

The aim of this section is to obtain a regularity result for flat free boundaries in
the spirit of \cite[Theorem 8.1]{MR618549}, but for weak solutions of our
inhomogeneous one-phase problem instead. Such a result was proven in \cite{MR3667700}
in a more general setting but using a more restrictive notion of weak solution. Thus,
our goal will be to show that any weak solution \(u:B_1 \to \R\) in the sense of
Definition \ref{defn:1} is also a weak solution in the sense of \cite[Definition
2.2]{MR3667700}.
\begin{proposition}
  \label{prop:2}
  If \(u\) is a weak solution in \(B_1\), then for \(\cH^{n-1}\)-a.e.
  \(x_0 \in \partial_{\mathrm{red}}\{u > 0\} \cap B_1\) we have the following asymptotic
  development for \(u\) and \(x \to 0\):
  \[
    u(x_0 + x) = Q(x_0)(x\cdot\nu_u(x_0))_+ + o(x),
  \]
  where \(\nu_u(x_0)\) is the interior normal vector of \(\{u > 0\}\) at \(x_0\).
\end{proposition}
\begin{proof}
  The result can be proven with the same arguments found in the proof of
  \cite[Theorem 4.8, Remark 4.9]{MR618549}. For the convergence of blow-up sequences
  we can use Lemma \ref{lem:1}, and instead of the representation theorem of
  \cite{MR618549} we use the definition of weak solution directly.
\end{proof}
\begin{proposition}
  \label{prop:3}
  Let \(u\) be a weak solution in \(B_1\) and let
  \(x_0 \in \partial\{u > 0\} \cap B_1\). If there is a ball \(B \subset \{u = 0\}\) touching the
  boundary \(\partial\{u > 0\}\) at \(x_0\), then
  \[
    \limsup_{\substack{x \to x_0 \\ u(x) > 0}} \frac{u(x)}{\dist(x, B)} \geq Q(x_0).
  \]
\end{proposition}
\begin{proof}
  The proof can be carried out exactly as \cite[Lemma 4.10]{MR618549} with the only
  modification being that instead of using the representation theorem of
  \cite{MR618549}, we use condition \autoref{item:3} from the definition of weak
  solution. For the convergence of blow-up sequences we use Lemma \ref{lem:1}.
\end{proof}
One of the conditions of \cite[Definition 2.2]{MR3667700} consists on an upper bound
on \(|\nabla u|\) at free boundary points. Precisely, for
\(x_0 \in \partial\{u > 0\} \cap \Omega\) we need to prove that
\[
  \limsup_{\substack{x \to x_0 \\ u(x) > 0}} |\nabla u(x)| \leq Q(x_0).
\]
To show this, we will adapt \cite[Theorem 6.3]{MR618549} to our inhomogeneous
setting, which gives an upper bound of \(|\nabla u|\) near the free boundary in terms of
\(Q\) and some error terms. To prove this result we will need the following two
lemmas.
\begin{lemma}
  \label{lem:2}
  Let \(\delta,\kappa > 0\) be small constants and assume \(u\) is a weak solution such
  that
  \[
    |u(x) - Q(x_0)((x-x_0)\cdot\nu_u(x_0))_+| \leq \delta r
  \]
  for \(x \in B_{r/\kappa}(x_0) \subset B_1\) with
  \(x_0 \in \partial_{\mathrm{red}}\{u > 0\}\). Then, there exist constants \(C_1\) and
  \(C_2\) depending only on \(n\) and \(\|f\|_{\infty}\) such that
  \[
    |\nabla u(x) - Q(x_0)\nu_u(x_0)| \leq C_1Q(x_0)(\delta + \kappa)
  \]
  for \(x \in B_{r/2\kappa}(x_0) \cap \{u = r\}\) and
  \[
    \cH^{n-1}(B_{r/2\kappa}(x_0) \cap \{u = r\}) \leq C_2\left( \frac{r}{\kappa} \right)^{n-1}.
  \]
\end{lemma}
\begin{proof}
  Without loss of generality we may assume \(r = \kappa\), \(\nu_u(x_0) = e_n\),
  \(Q(x_0) = 1\) and \(x_0 = 0\). By hypothesis, \(u\) is positive in
  \(B_1 \cap \{x_n > \delta \kappa\}\), and thus, the function
  \(v(x) \coloneqq u(x) - x_n\) solves \(-\Delta v = f\) in this region and satisfies
  \(|v| \leq \delta \kappa\). Hence, if \(\delta\) and \(\kappa\) are small enough, the estimate
  \begin{equation}
    \label{eq:grad-estimate}
    |\nabla v| \leq C_n\left( \frac{1}{\kappa}\|v\|_{\infty} + \kappa\|f\|_{\infty} \right) \leq C(\delta + \kappa)
    \quad\text{in}\quad B_{1/2} \cap \left\{ x_n > \frac{\kappa}{2} \right\}
  \end{equation}
  holds. In particular, since \(u(x) = \kappa\) implies
  \(x_n \geq (1-\delta)\kappa > \frac{\kappa}{2}\), we obtain
  \begin{equation}
    \label{eq:nth-deriv-estimate}
    |\partial_nu - 1| \leq C(\delta + \kappa) \quad\text{on}\quad B_{1/2} \cap \{u = \kappa\}
  \end{equation}
  and deduce that \(\partial_nu > 0\) on \(B_{1/2} \cap \{u = \kappa\}\) provided
  \(\delta\) and \(\kappa\) are small enough. It follows by the implicit function theorem that
  \(B_{1/2} \cap \{u = \kappa\}\) is a graph, that is, there exists a function
  \(g:B_{1/2}' \to \R\) such that \(u(x',g(x')) = \kappa\) for \(x' \in B_{1/2}'\), where
  \(x = (x',x_n) \in \R^{n-1} \times \R\) and \(B_{1/2}'\) is the ball of radius
  \(1/2\) in \(\R^{n-1}\). Moreover,
  \[
    \cH^{n-1}(B_{1/2} \cap \{u = \kappa\}) = \int_{B_{1/2}'} \sqrt{1 + |\nabla g|^2} dx'
  \]
  and
  \(\nabla g = -\frac{1}{\partial_nu}(\partial_1u,\ldots,\partial_{n-1}u)\). Now, using
  \eqref{eq:grad-estimate} and \eqref{eq:nth-deriv-estimate} it is easy to see
  that \(|\nabla g|^2 \leq C\) and therefore
  \[
    \cH^{n-1}(B_{1/2} \cap \{u = \kappa\}) \leq C|B_{1/2}'|
  \]
  which concludes the proof.
\end{proof}
\begin{lemma}
  \label{lem:3}
  Let \(u\) be a weak solution in \(B_1\) and let \(D \subset\subset B_1\) be a connected domain
  containing a free boundary point. There exists a constant
  \(C = C(n,D,c_D,C_D,\|f\|_{L^{\infty}(D)}) < \infty\) such that for
  \(\varepsilon > 0\) and \(B_r \subset D\) with \(r \leq 1\)
  \[
    \int_{B_r \cap \{u = \varepsilon\}} |\nabla u|d\cH^{n-1} \leq Cr^{n-1}.
  \]
\end{lemma}
\begin{proof}
  Since \(\nabla u\) is bounded by Proposition \ref{prop:1}, for almost all \(r\)
  we have
  \begin{align*}
    \int_{B_r \cap \{u = \varepsilon\}} |\nabla u|d\cH^{n-1} &= -\int_{B_r \cap \partial\{u > \varepsilon\}} \nabla u\cdot\nu d\cH^{n-1} \\
                                        &= \int_{\partial B_r \cap \{u > \varepsilon\}} \nabla u\cdot\nu d\cH^{n-1} + \int_{B_r \cap \{u > \varepsilon\}} f dx \\
                                        &\leq c\int_{\partial B_r} d\cH^{n-1} + \|f\|_{L^{\infty}(D)}\int_{B_r} dx.
  \end{align*}
  Thus, using that \(r \leq 1\) we see that
  \[
    \int_{B_r \cap \{u = \varepsilon\}} |\nabla u|d\cH^{n-1} \leq C(r^{n-1} + r^n) \leq Cr^{n-1}.
  \]
\end{proof}
\begin{proposition}
  \label{prop:4}
  If \(u\) is a weak solution in \(B_1\) with \(f \geq 0\), then for every open set
  \(D \subset\subset B_1\) there are constants \(C_1 = C_1(n,D,c_D,C_D)\) and
  \(C_2 = C_2(n,D,\|\nabla f\|_{\infty})\) such that
  \begin{equation}
    \label{eq:grad_bound}
    \sup_{B_r(x_0)} |\nabla u| \leq \sup_{B_R(x_0)} Q + C_1\left(\frac{r}{R}\right)^{\alpha} + C_2R^2
  \end{equation}
  for balls \(B_r(x_0) \subset B_R(x_0) \subset D\), provided \(r\) and \(R\) are
  sufficiently small.
\end{proposition}
\begin{proof}
  The proof of this result is for the most part the same as the proof of
  \cite[Theorem 6.3]{MR618549}, with the exception of some small modifications to
  account for the fact that \(-\Delta u = f\) in \(\{u > 0\}\). Below we mainly highlight
  this modifications. For more details on the rest of the steps we refer the reader
  to \cite{MR618549}.

  For any domain \(D \subset\subset B_1\) with smooth boundary define
  \[
    U \coloneqq (|\nabla u| - Q_D)_+ \quad\text{with}\quad Q_D \geq \sup_D Q.
  \]
  Then \(U\) is bounded and an elementary computation shows that \(U\) is a
  subsolution in \(\{u > 0\}\) of the equation \(-\Delta v = g\) with
  \(g = -\frac{\nabla u}{|\nabla u|}\cdot\nabla f\), that is, \(-\Delta U \leq g\) in
  \(\{u > 0\}\). Since we expect \(U\) to have zero boundary value on
  \(\partial\{u > 0\} \cap D\), first we prove the representation formula
  \[
    U(x) =
    \int_{\{u > 0\} \cap \partial D} U\partial_{\nu}G_x^0 d\cH^{n-1} - \int_{\{u > 0\} \cap D} G_x^0 d\lambda - \int_{\{u > 0\} \cap D} gG_x^0 dy
  \]
  for \(x \in \{u > 0\} \cap D\), where \(\lambda \coloneqq \Delta U + g\) is a Radon measure,
  \(G_x^0\) is the Green function of the domain \(\{u > 0\} \cap D\) and \(\nu\) is the
  inward normal vector to \(D\).
  \begin{enumerate}[label=\emph{Step \arabic*.},wide]
  \item Given \(x \in D\) with \(u(x) > \varepsilon\) and \(\varepsilon > 0\) small, define the Green
    function \(G_x^{\varepsilon}\) as follows:
    \begin{align*}
      &G_x^{\varepsilon} - \Gamma_x \in H^1(D) \quad\text{where } \Gamma_x \text{ is the fundamental solution,} \\
      &G_x^{\varepsilon} = 0 \quad\text{on } \partial D \text{ and } \{u \leq \varepsilon\} \cap D, \qquad \int_D \nabla G_x^{\varepsilon}\cdot\nabla\zeta = \zeta(x)
    \end{align*}
    for every \(\zeta \in H^1(D)\) that is Lipschitz continuous near \(x\) and zero on
    \(\partial D\) and \(\{u \leq \varepsilon\} \cap D\). Proceeding as in \cite{MR618549}, we see that
    \(G_x^{\varepsilon}\) is a positive function in \((\{u > \varepsilon\} \cap D) \setminus \{x\}\) satisfying that
    \[
      G_x^{\varepsilon} \leq G_x^{\varepsilon'} \quad\text{for}\quad 0 < \varepsilon' \leq \varepsilon,
    \]
    and \(G_x^0 \leq G_x\) where \(G_x\) is the Green function with respect to the
    entire domain \(D\). Now, choose \(\varepsilon_x\), \(d_x\) and \(C_x\) such that
    \[
      G_x^0 \leq C_x\frac{u(x)}{2} \text{ on } \partial B_{d_x}(x),
      \quad\text{and}\quad u - \varepsilon_x \geq \frac{u(x)}{2} \text{ in } B_{d_x}(x).
    \]
    Then, for \(\varepsilon \leq \varepsilon_x\) we can define the test function
    \(\zeta\) as \(\zeta = \max(G_x^{\varepsilon} - C_x(u-\varepsilon)_+,0)\) in
    \(D \setminus B_{d_x}(x)\) and \(\zeta = 0\) in \(D \cap B_{d_x}(x)\). Using
    \(\zeta\) as test function for \(G_x^{\varepsilon}\) we find
    \[
      \int_D \nabla G_x^{\varepsilon}\cdot\nabla\zeta = \zeta(x) = 0.
    \]
    Also, since \(\zeta = 0\) in \(\{u \leq \varepsilon\} \cap D\) by construction, testing
    \(u\) against \(\zeta\) we see that
    \[
      \int_D \nabla u\cdot\nabla\zeta = \int_D f\zeta.
    \]
    Combining the two equalities and using that \(f,\zeta \geq 0\) yields
    \[
      \int_{D \setminus B_{d_x}(x)} \nabla(G_x^{\varepsilon} - C_x(u-\varepsilon))\cdot\nabla\zeta = -C_x\int_D f\zeta \leq 0.
    \]
    Hence, if we set
    \(A = \{y \in D \setminus B_{d_x}(x) \mid G_x^{\varepsilon} - C_x(u-\varepsilon)_+ > 0\}\), then
    \(A \subset \{u > \varepsilon\}\) and so
    \[
      \int_A |\nabla(G_x^{\varepsilon} - C_x(u-\varepsilon))|^2 =
      \int_{D \setminus B_{d_x}(x)} \nabla(G_x^{\varepsilon} - C_x(u-\varepsilon))\cdot\nabla\zeta \leq 0.
    \]
    This would imply that \(G_x^{\varepsilon} - C_x(u-\varepsilon)_+ = 0\) in \(A\), contradicting the
    definition of \(A\). Thus, we conclude that \(A\) is empty, that is,
    \[
      G_x^{\varepsilon} \leq C_x(u-\varepsilon)_+ \quad\text{in}\quad D \setminus B_{d_x}(x).
    \]
    From this point the rest of this step can be carried out in the same way as
    \cite{MR618549}.
  \item As in \cite{MR618549}, we get the same representation formula for
    \(\varepsilon > 0\) but with an additional term:
    \[
      U(x) = \int_{\{u > \varepsilon\} \cap \partial D} U\partial_{\nu}G_x^{\varepsilon} d\cH^{n-1}
      - \int_{\{u > \varepsilon\} \cap D} G_x^{\varepsilon} d\lambda + \int_{\{u > \varepsilon\} \cap D} gG_x^{\varepsilon} dy
      + \int_{\{u = \varepsilon\} \cap \overline{D}} q_x^{\varepsilon}U d\cH^{n-1}.
    \]
    The first and second integrands converge to the desired limit as in
    \cite{MR618549}. Also, since
    \(\|g\|_{L^{\infty}(D)} \leq \|\nabla f\|_{L^{\infty}(D)} < \infty\), the third integrand converges to
    \[
      \int_{\{u > 0\} \cap D} gG_x^0 dy
    \]
    as wanted. Thus, we have to show that
    \[
      \int_{\{u = \varepsilon\} \cap \overline{D}} U d\cH^{n-1} \to 0
    \]
    as \(\varepsilon \to 0\).
  \item Let \(\kappa > 0\) be a small constant and for \(\delta > 0\) choose an open set
    \(D_{\delta}\) such that
    \[
      \overline{D} \subset D_{\delta} \subset\subset \Omega \quad\text{and}\quad \cH^{n-1}(\partial\{u > 0\} \cap (D_{\delta} \setminus \overline{D})) \leq \delta.
    \]
    Also, for \(\varepsilon > 0\) let \(K_{\varepsilon}\) be the set of points
    \(x_0 \in \partial_{\mathrm{red}}\{u > 0\} \cap \overline{D}\) such that
    \[
      |u(x_0 + x) - Q(x_0)(x\cdot\nu_u(x_0))_+| \leq \delta\kappa|x| \quad\text{for } x \in B_{\varepsilon/\kappa}.
    \]
    The sets \(K_{\varepsilon}\) are monotone in \(\varepsilon\) and by Proposition \ref{prop:2} they
    form a covering of the reduced free boundary up to a set of \(\cH^{n-1}\) measure
    zero. Hence by Remark \ref{rem:5} we can choose
    \(\varepsilon = \varepsilon(\delta)\) small enough such that
    \[
      \cH^{n-1}(\partial\{u > 0\} \cap (\overline{D} \setminus K_{\varepsilon})) \leq \delta \quad\text{and}\quad
      B_{\varepsilon/\kappa}(\overline{D}) \subset D_{\delta},
    \]
    where
    \(B_{\varepsilon/\kappa}(\overline{D}) = \bigcup_{x \in \overline{D}} B_{\varepsilon/\kappa}(x)\). Now, choose a finite
    covering of \(K_{\varepsilon}\) with balls \(B_{\varepsilon}^j\) of radius \(\varepsilon\) such that
    \[
      \chi_{K_{\varepsilon}} \leq \sum_j \chi_{B_{2\varepsilon}^j} \leq C(n),
    \]
    and let \(x_j \in B_{\varepsilon}^j \cap K_{\varepsilon}\). As in \cite{MR618549}, we can use Lemma
    \ref{lem:2} and \cite[Theorem 2.1]{MR3667700} to obtain the estimate
    \[
      \sum_j\int_{B_{\varepsilon/(2\kappa)}(x_j) \cap \{u = \varepsilon\}} U d\cH^{n-1} \leq
      C(\kappa)\delta\cH^{n-1}(\partial\{u > 0\} \cap D_{\delta})
    \]
    and conclude that the left hand side tends to zero as \(\delta \to 0\). Thus, it remains
    to estimate
    \[
      \int_{L_{\varepsilon}} U d\cH^{n-1}, \quad\text{where }
      L_{\varepsilon} \coloneqq (\overline{D} \cap \{u = \varepsilon\}) \setminus \bigcup_j B_{\varepsilon/(2\kappa)}(x_j).
    \]

    For \(x \in L_{\varepsilon}\) with \(U(x) > 0\) we have
    \(|\nabla u(x)| \geq c > 0\). Then for the maximum ball \(B_r(x)\) in
    \(\{u > 0\}\) we can use the following approximate mean value formula for \(u\)
    \[
      u(x) = \fint_{\partial B_r(x)} u d\cH^{n-1} - \int_{B_r(x)} f(y)G_{B_r(x)}(x,y) dy,
      \quad\text{(\(G_{B_r(x)}\) is the Green function of \(B_r(x)\))}
    \]
    to obtain the estimate
    \begin{align*}
      c \leq |\nabla u(x)| &\leq \frac{C_n}{r}\fint_{\partial B_r(x)} u d\cH^{n-1} + C_nr\|f\|_{L^{\infty}(B_r(x))} \\
                   &= \frac{C_n}{r}\left(u(x) + \int_{B_r(x)} f(y)G_{B_r(x)}(x,y)dy\right) + C_nr\|f\|_{L^{\infty}(B_r(x))}.
    \end{align*}
    Since
    \[
      \int_{B_r(x)} f(y)G_{B_r(x)}(x,y)dy
      \leq \|f\|_{L^{\infty}(B_r(x))}\int_{B_r(x)} G_{B_r(x)}(x,y)dy
      = C_n\|f\|_{L^{\infty}(B_r(x))}r^2,
    \]
    it follows that
    \begin{equation}
      \label{eq:radius-bound}
      \begin{aligned}
        c &\leq \frac{C_n}{r}\left(u(x) + \int_{B_r(x)} f(y)G_{B_r(x)}(x,y)dy\right)
            + C_nr\|f\|_{L^{\infty}(B_r(x))} \\
          &\leq \frac{C_n\varepsilon}{r} + C_n\|f\|_{L^{\infty}(D)}r.
      \end{aligned}
    \end{equation}
    Suppose \(C_n\|f\|_{\infty}r \geq \frac{C_n\varepsilon}{r}\). Then
    \(c \leq 2C_n\|f\|_{\infty}r\) and therefore \(r \geq c_1 > 0\). Set
    \(v \coloneqq -\frac{\nabla u(x)}{|\nabla u(x)|}\) and write
    \[
      u(x + tv) - u(x) = \int_0^t \nabla u(x + sv)\cdot v ds.
    \]
    Since \(-\Delta u = f\) in \(B_{c_1}(x) \subset B_r(x)\), by Schauder estimates
    \begin{align*}
      \|\nabla u\|_{C^{0,\alpha}(B_{c_1/2}(x))} &\leq C_n\left(\frac{1}{c_1}\|u\|_{L^{\infty}(B_{c_1}(x))} + c_1\|f\|_{L^{\infty}(B_{c_1}(x))}\right) \\
                                  &\leq C_n\left(\frac{1}{c_1}\|u\|_{L^{\infty}(D)} + c_1\|f\|_{L^{\infty}(D)}\right).
    \end{align*}
    Thus, the function \(h(s) \coloneqq \nabla u(x+sv)\cdot v\) is H\"older continuous for
    \(0 \leq s < \frac{c_1}{2}\) and its H\"older constant is independent of \(x\).
    Hence, since \(h(0) = -|\nabla u(x)| \leq -c\), there exists some \(s_0\) independent of
    \(x\) such that \(h(s) \leq -\frac{c}{2}\) for \(s \in [0,s_0)\). It follows that for
    \(t < s_0\) we have
    \[
      u(x + tv) - u(x) = \int_0^t \nabla u(x + sv)\cdot v ds \leq -\frac{c}{2}t.
    \]
    Consequently, taking \(\varepsilon\) small enough we can guarantee the existence of some
    \(t\) such that \(2\varepsilon/c < t < s_0 < c_1\) and
    \[
      u(x+tv) \leq \varepsilon - \frac{c}{2}t < 0.
    \]
    Since this contradicts the non-negativity of \(u\), we deduce that
    \(C_n\|f\|_{\infty}r \leq \frac{C_n\varepsilon}{r}\). Inserting this estimate back in
    \eqref{eq:radius-bound} we obtain
    \[
      c \leq \frac{2C_n\varepsilon}{r},
    \]
    which means that, for a small constant \(\tau\) independent of \(\varepsilon\),
    \(\kappa\) and \(x\), the ball \(B_{\varepsilon/\tau}(x)\) intersects the free boundary. From this
    point the remaining part of this step can be finished as in \cite{MR618549} using
    Lemma \ref{lem:3} and \cite[Theorem 2.1]{MR3667700}.
  \item This step is for the most part the same as \cite{MR618549}. Let \(x_0\) be a
    free boundary point and set \(D = B_r(x_0)\). Let \(\nu\) be the inward normal
    vector to \(B_r(x_0)\). We will prove the following estimate for the Poisson
    kernel \(\partial_{\nu}G_x^0\):
    \[
      \int_{\{u > 0\} \cap \partial B_r(x_0)} \partial_{\nu}G_x^0 d\cH^{n-1} \leq 2^{-\alpha}
    \]
    for \(x \in B_{r/2}(x_0)\) and some \(\alpha > 0\), provided \(r\) is small enough. By
    the same computations in \cite{MR618549} we have
    \[
      \int_{\partial B_r(x_0) \cap \{u > \varepsilon\}} \partial_{\nu}G_x^{\varepsilon} d\cH^{n-1}
      \leq 1 - \limsup_{\delta \to 0} \frac{1}{\delta}\int_{\partial B_{r-\delta}(x_0)} v_x^{\varepsilon} d\cH^{n-1}
      \eqqcolon 1 - s_{\varepsilon},
    \]
    and we want to estimate \(s_{\varepsilon}\) from below. Define
    \[
      u_{\varepsilon} \coloneqq (u-\varepsilon)_+, \qquad \Lambda_{\varepsilon} \coloneqq B_{r/2}(x_0) \cap \{u \leq \varepsilon\},
    \]
    and let \(w_{\varepsilon} \in H^1(B_r(x_0))\) be the solution of
    \[
      \begin{cases}
        \Delta w_{\varepsilon} = 0 \quad&\text{in } B_r(x_0) \setminus \Lambda_{\varepsilon}, \\
        w_{\varepsilon} = 1 \quad&\text{on } \partial B_r(x_0), \\
        w_{\varepsilon} = 0 \quad&\text{a.e. in } \Lambda_{\varepsilon}.
      \end{cases}
    \]
    Proceeding as in \cite{MR618549} we find that
    \[
      s_{\varepsilon} \geq cr^{2-n}\int_{\partial B_r(x_0)} \partial_{-\nu}w_{\varepsilon} d\cH^{n-1} \geq cr^{2-n}\lambda_{w_{\varepsilon}}(B_{r/2}(x_0)),
    \]
    where \(\lambda_{w_{\varepsilon}} \coloneqq \Delta w_{\varepsilon} \geq 0\) is a Radon measure. Moreover, we also
    see that \(u_{\varepsilon} \leq Crw_{\varepsilon}\) in \(B_r(x_0)\). Therefore, if we choose a sequence
    \(\varepsilon \to 0\) such that \(\{u = \varepsilon\} \cap B_{r/2}(x_0)\) are smooth surfaces, then
    \(w_{\varepsilon}\) is continuous in \(B_{r/2}(x_0)\) and zero in
    \(\Lambda_{\varepsilon}\). Also,
    \(\Delta(Crw_{\varepsilon} - u_{\varepsilon}) = f\) in
    \(B_{r/2}(x_0) \setminus \Lambda_{\varepsilon}\) and we can apply Lemma 2.1 of \cite{MR3667700} to
    conclude that
    \[
      \lambda_{Crw_{\varepsilon} - u_{\varepsilon}} \coloneqq \Delta(Crw_{\varepsilon} - u_{\varepsilon}) - f\chi_{B_{r/2}(x_0) \setminus \Lambda_{\varepsilon}}
    \]
    is a Radon measure in \(B_{r/2}(x_0)\). Hence,
    \[
      Cr\lambda_{w_{\varepsilon}} - \lambda_{u_{\varepsilon}} = \lambda_{Crw_{\varepsilon} - u_{\varepsilon}} \geq 0
    \]
    and we obtain
    \[
      \limsup_{\varepsilon \to 0} \lambda_{w_{\varepsilon}}(B_{r/2}(x_0))
      \geq \frac{c}{r}\limsup_{\varepsilon \to 0} \lambda_{u_{\varepsilon}}(B_{r/2}(x_0))
      \geq \frac{c}{r}\lambda_u(B_{r/2}(x_0)).
    \]
    Using Remark \ref{rem:1}, \cite[Lemma 2.4]{MR3667700} implies that
    \[
      \lambda_u(B_{r/2}(x_0)) \geq cr^{n-1}
    \]
    for \(r\) small enough. Thus,
    \[
      \limsup_{\varepsilon \to 0} s_{\varepsilon}
      \geq cr^{2-n}\limsup_{\varepsilon \to 0} \lambda_{w_{\varepsilon}}(B_{r/2}(x_0))
      \geq cr^{1-n}\lambda_u(B_{r/2}(x_0))
      \geq c
    \]
    and this finishes the estimate for the Poisson kernel.
  \item As in the previous step, set \(D = B_r(x_0)\) with \(r \leq R\) and
    define
    \[
      \phi(r) \coloneqq \sup_{B_r(x_0)} U(x),
      \qquad Q_D \coloneqq \sup_{B_R(x_0)} Q(x) \geq \sup_{B_r(x_0)} Q(x).
    \]
    Then for \(x \in B_{r/2}(x_0)\) we get, provided \(r\) is small enough,
    \begin{align*}
      U(x) &= \int_{\{u > 0\} \cap \partial B_r(x_0)} U\partial_{\nu}G_x^0 d\cH^{n-1} - \int_{\{u > 0\} \cap B_r(x_0)} G_x^0 d\lambda + \int_{\{u > 0\} \cap B_r(x_0)} gG_x^0 dy \\
           &\leq \left(\sup_{B_r(x_0)} U(x)\right)\int_{\{u > 0\} \cap \partial B_r(x_0)} \partial_{\nu}G_x^0 d\cH^{n-1} + \|g\|_{L^{\infty}(B_r(x_0))}\int_{\{u > 0\} \cap B_r(x_0)} G_x^0 dy \\
           &\leq 2^{-\alpha}\phi(r) + \|\nabla f\|_{L^{\infty}(B_r(x_0))}\int_{\{u > 0\} \cap B_r(x_0)} G_x^0 dy \\
           &\leq 2^{-\alpha}\phi(r) + Cr^2.
    \end{align*}
    Here we have used: in the first step, the representation formula we have
    proved for \(U\); in the second step, the fact that
    \(\int_{\{u > 0\} \cap B_r(x_0)} G_x^0 d\lambda\) is non-negative; in the third step,
    the estimate for the Poisson kernel; and in the last step, the fact
    that
    \[
      \int_{\{u > 0\} \cap B_r(x_0)} G_x^0 dy \leq \int_{B_r(x_0)} G_x dy \leq C_nr^2.
    \]
    Thus, we have obtained
    \[
      \phi\left(\frac{r}{2}\right) \leq 2^{-\alpha}\phi(r) + Cr^2,
    \]
    and an iteration of this estimate yields
    \[
      \phi\left(\frac{r}{2^k}\right)
      \leq 2^{-k\alpha}\phi(r) + Cr^2\sum_{j=1}^k \frac{1}{4^{k-j}}
      = 2^{-k\alpha}\phi(r) + C\left(1 - \frac{1}{4^k}\right)r^2
      \leq 2^{-k\alpha}\phi(r) + Cr^2
    \]
    for any \(r \leq R\) small enough. In particular, if \(R\) is also
    sufficiently small we can apply this estimate with \(r = R\) to obtain
    \[
      \phi\left(\frac{R}{2^k}\right)
      \leq 2^{-k\alpha}\phi(R) + CR^2.
    \]
    Now, given \(r \leq \frac{R}{2}\), choose \(k\) such that
    \[
      \frac{R}{2^{k+1}} \leq r \leq \frac{R}{2^k}.
    \]
    Then, since \(\phi(r)\) is monotone in \(r\),
    \[
      \phi(r) \leq \phi\left(\frac{R}{2^k}\right)
      \leq 2^{-k\alpha}\phi(R) + CR^2
      \leq \left(\frac{2r}{R}\right)^{\alpha}\phi(R) + CR^2.
    \]
    Therefore,
    \[
      \sup_{B_r(x_0)} |\nabla u| - \sup_{B_R(x_0)} Q
      \leq \phi(r)
      \leq \phi(R)\left(\frac{2r}{R}\right)^{\alpha} + CR^2
    \]
    and the proof is finished.\qedhere
  \end{enumerate}
\end{proof}
\begin{corollary}
  \label{cor:1}
  Let \(u\) be a weak solution. Then
  \[
    \limsup_{\substack{x \to x_0 \\ u(x) > 0}} |\nabla u(x)| \leq Q(x_0)
  \]
  for every \(x_0 \in \partial\{u > 0\} \cap \Omega\).
\end{corollary}
\begin{proof}
  The result follows immediately from Proposition \ref{prop:4} by taking the limits
  \(r \to 0\) and \(R \to 0\) in \eqref{eq:grad_bound}.
\end{proof}

As a consequence of Remark \ref{rem:1}, Propositions \ref{prop:2} and \ref{prop:3},
and Corollary \ref{cor:1} we obtain that weak solutions as defined in Definition
\ref{defn:1} are also weak solutions in the sense of \cite{MR3667700}. In particular,
we get a regularity result for flat free boundaries directly from \cite[Theorem
4.3]{MR3667700}.
\begin{definition}
  Let \(0 < \sigma_-,\sigma_+ < 1\), \(\tau > 0\) and \(\nu\) be a unit vector. We say
  \(u\) is of class
  \[
    F(\sigma_+,\sigma_-;\tau) \text{ in } B_{\rho}(x_0) \text{ in the direction } \nu,
  \]
  if
  \begin{enumerate}[label=(\roman*)]
  \item \(u\) is a weak solution in \(B_{\rho}(x_0)\).
  \item \(x_0 \in \partial\{u > 0\} \cap \Omega\) and
    \begin{alignat*}{2}
      u(x) &\geq Q(x_0)((x-x_0)\cdot\nu - \sigma_+\rho) \quad&&\text{for}\quad (x-x_0)\cdot\nu \ge \sigma_+\rho, \\
      u(x) &= 0 \quad&&\text{for}\quad (x-x_0)\cdot\nu \le -\sigma_-\rho.
    \end{alignat*}
  \item \(|\nabla u| \le Q(x_0)(1+\tau)\) in \(B_{\rho}(x_0)\) and
    \(\osc_{B_{\rho}(x_0)} Q \leq Q(x_0)\tau\).
  \end{enumerate}
\end{definition}
\begin{theorem}
  \label{thm:4}
  There are constants \(\alpha > 0\), \(\sigma_0 > 0\) and \(C > 0\) such that the
  following holds: if
  \[
    u \in F(1,\sigma;\infty) \text{ in } B_{\rho}(x_0) \text{ in the direction } \nu
  \]
  with \(\sigma \le \sigma_0\) and \(\rho\) small enough, then
  \[
    \partial\{u > 0\} \cap B_{\rho/4}(x_0) \text{ is a } C^{1,\alpha} \text{ surface,}
  \]
  more precisely, a graph in the direction \(\nu\) of a \(C^{1,\alpha}\) function,
  and for \(x,y\) on this surface
  \[
    |\nu(x) - \nu(y)| \le C\sigma\left|\frac{x - y}{\rho}\right|^{\alpha}.
  \]
  The constants \(\alpha\), \(\sigma_0\) and \(C\) depend only on \(n\),
  \(\|f\|_{W^{1,\infty}(\Omega)}\), \(Q_{\mathrm{min}}\), \(Q_{\mathrm{max}}\) and
  \(\|\nabla Q\|_{L^{\infty}(\Omega)}\).
\end{theorem}
As a consequence of Theorem \ref{thm:4} we obtain the smoothness of the free boundary
in a neighborhood of any regular point.
\begin{definition}
  Let \(u\) be a weak solution in \(B_1\). We say a point
  \(x_0 \in \partial\{u > 0\} \cap B_1\) is a regular point if there exists a blow-up limit of
  \(u\) at \(x_0\) of the form \(u_0(x) = Q(x_0)(x\cdot\nu)_+\) for some unit vector
  \(\nu \in \R^n\).
\end{definition}
\begin{corollary}
  \label{cor:2}
  Let \(u\) be a weak solution in \(B_1\) and let
  \(x_0 \in \partial\{u > 0\} \cap B_1\) be a regular point. Then,
  \(\partial\{u > 0\}\) is a \(C^{\infty}\)-surface in a neighborhood of \(x_0\).
\end{corollary}
\begin{proof}
  Replacing \(u\) with \(Q(x_0)^{-1}u\) we may assume that \(Q(x_0) = 1\). Let
  \(u_{r_k}\) be a blow-up sequence at \(x_0\) converging (in the sense of
  Lemma \ref{lem:1}) to \(u_0(x) = (x\cdot\nu)_+\) for some unit vector
  \(\nu \in \R^n\). Since \(u_{r_k} \to u_0\) locally uniformly in \(\R^n\) and
  \(\partial\{u_{r_k} > 0\} \to \partial\{u_0 > 0\}\) locally in Hausdorff distance, given any
  small \(\varepsilon > 0\) we have
  \begin{gather}
    \|u_{r_k} - u_0\|_{L^{\infty}(B_1)} \leq \varepsilon, \label{eq:3} \\
    \partial\{u_{r_k} > 0\} \subset \{-\varepsilon \leq x\cdot\nu \leq \varepsilon\} \cap B_1, \label{eq:4}
  \end{gather}
  for \(k\) large enough. On the one hand, it follows from \eqref{eq:3} that
  \[
    (x\cdot\nu)_+ - \varepsilon \leq u_{r_k}(x) \leq (x\cdot\nu)_+ + \varepsilon \quad\text{in } B_1
  \]
  and therefore
  \[
    u_{r_k} > 0 \quad\text{in } \{x\cdot\nu > \varepsilon\} \cap B_1, \qquad
    u_{r_k} \leq \varepsilon \quad\text{in } \{x\cdot\nu < -\varepsilon\} \cap B_1.
  \]
  On the other hand, by \eqref{eq:4} we know that in
  \(\{x\cdot\nu < -\varepsilon\} \cap B_1\) either \(u_{r_k} = 0\) or
  \(u_{r_k} > 0\). Fix any ball
  \(B_r(x) \subset \{x\cdot\nu < -\varepsilon\} \cap B_1\). If \(u_{r_k} > 0\) in
  \(B_r(x)\), then by the non-degeneracy of \(u\) we can find a point
  \(y \in B_r(x)\) such that \(u_{r_k}(y) \geq cr\). Hence,
  \[
    cr \leq u_{r_k}(y) \leq \varepsilon
  \]
  and taking \(\varepsilon\) small enough we reach a contradiction. This proves that
  \[
    u_{r_k} = 0 \quad\text{in } \{x\cdot\nu < -\varepsilon\} \cap B_1.
  \]
  In particular \(u_{r_k} \in F(1,\varepsilon;\infty)\) in \(B_1\) in the direction
  \(\nu\) (in fact \(u_{r_k} \in F(\varepsilon,\varepsilon;\infty)\)) and therefore
  \(u \in F(1,\varepsilon;\infty)\) in \(B_{r_k}(x_0)\). Hence, taking
  \(\varepsilon\) and \(r_k\) small enough we can apply Theorem \ref{thm:4} to conclude that
  \(\partial\{u > 0\} \cap B_{\rho}(x_0)\) is a \(C^{1,\alpha}\)-surface for some
  \(0 < \alpha < 1\) and some small \(\rho\) independent of \(x_0\). Thus,
  \[
    \partial_{\mathrm{red}}\{u > 0\} \cap B_{\rho}(x_0) = \partial\{u > 0\} \cap B_{\rho}(x_0)
  \]
  and \cite[Corollary 1.6]{carducci2025regularity} implies that
  \(\partial\{u > 0\} \cap B_{\rho/2}(x_0)\) is of class \(C^{\infty}\).
\end{proof}

\section{Blow-up limits of weak solutions}
\label{sec:blow-ups-weak}

The blow-up procedure we saw in Lemma \ref{lem:1} applies, in particular, to
weak solutions. However, the blow-up limit \(u_0\) is not guaranteed to be a
weak solution. In fact, without any additional assumptions the limit of a
sequence of weak solutions need not be a weak solution. Indeed, the function
\[
  u_{\varepsilon}(x) \coloneqq \max(|x_n|,\varepsilon) - \varepsilon
\]
is a weak solution for \(f \equiv 0\), \(Q \equiv 1\) for all \(\varepsilon > 0\). However,
\(u_{\varepsilon}\) converges to \(u_0(x) = |x_n|\) in \(H^1_{\mathrm{loc}}(\R^n)\) as
\(\varepsilon \to 0\), and \(u_0\) is not a weak solution as we saw in Remark \ref{rem:4}. The
issue is that the \(n\)-dimensional Lebesgue measure of the contact set
\(\{u_{\varepsilon} = 0\}\) vanishes in the limit.

In the proof of \cite[Theorem 5.1]{MR1620644} it was shown that, for the classical
Alt-Caffarelli problem (\(f \equiv 0\), \(Q \equiv \text{const.}\)), a positive density
condition on the contact set \(\{u = 0\}\) guarantees that the blow-up limit of Lemma
\ref{lem:1} is a weak solution. We now extend this result to the more general setting
\eqref{eq:inhom-alt-caff-free-bound-form}. The proof follows the steps of
\cite{MR1620644} with minor adjustments.
\begin{proposition}
  \label{prop:5}
  Let \(u\) be a weak solution in \(B_1\) and assume \(u\) satisfies the following
  uniform positive density condition: for every open subset \(D \subset\subset B_1\), there
  exists a constant \(c_D > 0\) such that
  \begin{equation}
    \label{eq:densityEst}
    \frac{|\{u = 0\} \cap B_r(x)|}{|B_r|} \geq c_D > 0
  \end{equation}
  for any ball \(B_r(x) \subset D\) with \(x \in \partial\{u > 0\} \cap D\). Then, for any
  \(x_0 \in \partial\{u > 0\} \cap B_1\) and any blow-up sequence at \(x_0\)
  \[
    u_{r_k}(x) = \frac{1}{r_k}u(x_0 + r_kx), \quad r_k \to 0,
  \]
  there exists a Lipschitz continuous blow-up limit \(u_0:\R^n \to \R\) such
  that, up to subsequence,
  \begin{enumerate}
  \item \(u_{r_k} \to u_0\) in the sense of Lemma \ref{lem:1},
  \item \(u_0\) is a weak solution in \(\R^n\) with \(f \equiv 0\) and
    \(Q \equiv Q(x_0)\).
  \end{enumerate}
\end{proposition}
\begin{proof}
  Given a blow-up sequence \(u_{r_k}\), Lemma \ref{lem:1} implies the existence of a
  blow-up limit \(u_0 \in C^{0,1}(\R^n)\) such that \(u_{r_k} \to u_0\) (in the sense of
  Lemma \ref{lem:1}). Moreover, we also know from Lemma \ref{lem:1} that \(u_0\)
  satisfies conditions \autoref{item:1} and \autoref{item:2} of Definition
  \ref{defn:1} with \(f \equiv 0\). Let us prove that \(u_0\) satisfies condition
  \autoref{item:3}.
  \begin{enumerate}[label=\emph{Step \arabic*.},wide]
  \item First we claim that \eqref{eq:densityEst} also holds for the blow-up limit
    \(u_0\). Let \(D \subset\subset \R^n\) and let \(B_r(x) \subset D\) be a ball with
    \(x \in \partial\{u_0 > 0\}\). Since
    \(\partial\{u_{r_k} > 0\} \to \partial\{u_0 > 0\}\) locally in Hausdorff distance, we can find a
    sequence of points \(x_k \in \partial\{u_{r_k} > 0\}\) such that \(x_k \to x\). Notice that
    for any \(k\)
    \begin{align*}
      |\{u_{r_k} = 0\} \cap B_r(x_k)| &= \int_{B_r(x_k)} \chi_{\{u_{r_k} = 0\}}(y)dy
                                     = \frac{1}{r_k^n}\int_{B_{rr_k}(x_0 + r_kx_k)} \chi_{\{u = 0\}}(y)dy \\
                                   &= \frac{1}{r_k^n}|\{u = 0\} \cap B_{rr_k}(x_0 + r_kx_k)|.
    \end{align*}
    Since \(B_r(x_k) \subset D\) for \(k\) large enough, there exists an open set
    \(D' \subset\subset \Omega\) (depending on \(D\)) such that
    \(B_{rr_k}(x_0 + r_kx_k) \subset x_0 + r_kD \subset D'\). Moreover,
    \(x_k \in \partial\{u_{r_k} > 0\}\) implies that
    \(x_0 + r_kx_k \in \partial\{u > 0\}\). Hence, by the positive density of \(u\) at
    \(x_0 + r_kx_k\) there is a constant \(c > 0\) depending on \(D'\) (and therefore
    on \(D\)) such that
    \[
      \frac{1}{r_k^n}|\{u = 0\} \cap B_{rr_k}(x_0 + r_kx_k)| \geq
      \frac{1}{r_k^n}|B_{rr_k}|c =
      |B_r|c.
    \]
    Thus,
    \[
      \frac{|\{u_{r_k} = 0\} \cap B_r(x_k)|}{|B_r|} \geq c > 0
    \]
    and the convergence \(\chi_{\{u_{r_k} > 0\}} \to \chi_{\{u_0 > 0\}}\) in
    \(L^1_{\mathrm{loc}}(\R^n)\) implies that
    \[
      \frac{|\{u_0 = 0\} \cap B_r(x)|}{|B_r|} \geq c
    \]
    with \(c = c(D)\).
  \item Next we claim that
    \(\cH^{n-1}(\partial\{u_0 > 0\} \setminus \partial_{\mathrm{red}}\{u_0 > 0\}) = 0\). Assume by
    contradiction that the claim does not hold. Then by \cite[4.5.6 (3)]{MR257325},
    \[
      s^{1-n}\cH^{n-1}(\partial_{\mathrm{red}}\{u_0 > 0\} \cap B_{s}(x_1)) \to 0 \quad
      \text{as}\quad s \to 0
    \]
    for \(\cH^{n-1}\)-a.a. points
    \(x_1 \in \partial\{u_0 > 0\} \setminus \partial_{\mathrm{red}}\{u_0 > 0\}\). Let
    \(x_1\) be such a point and consider a second blow-up sequence
    \[
      u_{0s_l}(x) \coloneqq (u_0)_{s_l}(x) = \frac{1}{s_l}u_0(x_1 + s_lx).
    \]
    By Lemma \ref{lem:1}, \(u_{0s_l}\) converges to a blow-up limit \(u_{00}\).
    Moreover, the fact that \(\{u_0 > 0\}\) has locally finite perimeter in \(\R^n\)
    (recall Remark \ref{rem:2}) implies that for \(R \in (0,\infty)\) and
    \(\phi \in C^{\infty}_{\mathrm{c}}(B_R;\R^n)\)
    \begin{align*}
      \int_{B_R} \chi_{\{u_{0s_l} > 0\}}\Div \phi &= s_l^{1-n}\int_{B_{Rs_l}(x_1)} \chi_{\{u_0 > 0\}}\Div\left(\phi\left(\frac{x - x_1}{s_l}\right)\right) \\
                                                  &= s_l^{1-n}\int_{\partial_{\mathrm{red}}\{u_0 > 0\} \cap B_{Rs_l}(x_1)} \phi\left(\frac{x - x_1}{s_l}\right)\cdot\nu d\cH^{n-1} \\
                                                  &\leq \|\phi\|_{L^{\infty}}s_l^{1-n}\cH^{n-1}(\partial_{\mathrm{red}}\{u_0 > 0\} \cap B_{Rs_l}(x_1)).
    \end{align*}
    Taking the limit as \(l \to \infty\) we obtain that
    \[
      \int_{B_R} \chi_{\{u_{00} > 0\}}\Div \phi = 0, \quad \forall\phi \in C^{\infty}_{\mathrm{c}}(B_R;\R^n).
    \]
    Therefore, \(\chi_{\{u_{00} > 0\}}\) is a function of bounded variation which
    is constant a.e. in \(B_R\). Since \(u\) satisfies \eqref{eq:densityEst},
    so do \(u_0\) and \(u_{00}\) by the previous step. Thus, \(u_{00} = 0\) in
    \(B_R\) which contradicts the non-degeneracy of \(u_0\) at \(x_1\).
  \item Our next step is to prove that \(\partial\{u_0 > 0\}\) is a
    \(C^{1,\alpha}\)-surface in a neighborhood of any point of the reduced boundary
    \(\partial_{\mathrm{red}}\{u_0 > 0\}\). Let
    \(x_1 \in \partial_{\mathrm{red}}\{u_0 > 0\}\) and once again let
    \(u_{00}\) be a blow-up limit of \(u_0\) at \(x_1\). Then,
    \[
      \{u_{0s_l} = 0\} \longrightarrow \{x\cdot\nu < 0\} \quad\text{in } L^1_{\mathrm{loc}},
    \]
    where \(\nu = \nu_{u_0}(x_1)\) denotes the inward unit normal vector to
    \(\{u_0 > 0\}\) at \(x_1\). Since
    \[
      \chi_{\{u_{0s_l} > 0\}} \longrightarrow \chi_{\{u_{00} > 0\}} \quad\text{in } L^1_{\mathrm{loc}}(\R^n)
    \]
    as well, we infer that \(\{u_{00} = 0\} = \{x\cdot\nu < 0\}\) almost everywhere. By the
    continuity and the uniform density estimate of \(u_{00}\) it follows that
    \[
      \{u_{00} = 0\} = \{x\cdot\nu \leq 0\} \text{ everywhere}.
    \]
    Thus, \(\partial\{u_{00} > 0\} = \{x\cdot\nu = 0\}\) and the local convergence
    \(\partial\{u_{0s_l} > 0\} \to \partial\{u_{00} > 0\}\) in Hausdorff distance implies that for
    any \(\varepsilon > 0\)
    \[
      \|u_{0s_l} - u_{00}\|_{L^{\infty}(B_1)} \leq \varepsilon, \qquad
      \partial\{u_{0s_l} > 0\} \cap B_1 \subset \{-\varepsilon \leq x\cdot\nu \leq \varepsilon\} \cap B_1,
    \]
    for \(l\) large enough. In particular, in
    \(\{x\cdot\nu < -\varepsilon\} \cap B_1\) either \(u_{0s_l} > 0\) or
    \(u_{0s_l} = 0\). Moreover, \(u_{0s_l} \leq \varepsilon\) in
    \(\{x\cdot\nu < -\varepsilon\} \cap B_1\) since \(u_{00} = 0\). Now fix
    \(B_\rho(x) \subset \{x\cdot\nu < -\varepsilon\} \cap B_1\) and suppose
    \(u_{0s_l}\) is positive in \(B_\rho(x)\). Then, by the non-degeneracy of \(u_0\)
    there exists a point \(y \in B_{\rho}(x)\) such that
    \[
      c\rho \leq u_{0s_l}(y) \leq \varepsilon
    \]
    and taking \(\varepsilon\) small enough we reach a contradiction. Therefore we conclude
    that \(u_{0s_l} = 0\) in \(\{x\cdot\nu < -\varepsilon\} \cap B_1\). In other words,
    \(u_{0s_l} \in F(1,\varepsilon;\infty)\) in \(B_1\) in the direction \(\nu\), or equivalently,
    \(u_0 \in F(1,\varepsilon;\infty)\) in \(B_{s_l}(x_1)\) in the direction \(\nu\). Repeating this
    argument with the first blow-up sequence \(u_{r_k}\) we obtain that
    \(u_{r_k} \in F(1,2\varepsilon;\infty)\) in \(B_{s_l}(x_1)\) in the direction
    \(\nu\). Hence, taking \(\varepsilon\) and \(s_l\) small enough we can apply Theorem
    \ref{thm:4} to deduce that \(\partial\{u_{r_k} > 0\} \cap B_{\rho}(x_1)\) is a
    \(C^{1,\alpha}\)-surface for some small radius \(\rho\). Lastly, since the surfaces
    \(\partial\{u_{r_k} > 0\} \cap B_{\rho}(x_1)\) are uniformly bounded in
    \(C^{1,\alpha}\) (also by Theorem \ref{thm:4}), we conclude that
    \(\partial\{u_0 > 0\}\) must be a \(C^{1,\alpha}\)-surface in \(B_{\rho}(x_1)\).
  \item Let \(\eta \in C^{\infty}_{\mathrm{c}}(\R^n)\). Since
    \(\cH^{n-1}(\partial\{u_0 > 0\} \setminus \partial_{\mathrm{red}}\{u_0 > 0\}) = 0\) and
    \(\partial_{\mathrm{red}}\{u_0 > 0\}\) is open relative to
    \(\partial\{u_0 > 0\}\), we can find a finite covering
    \[
      \supp \eta \cap (\partial\{u_0 > 0\} \setminus \partial_{\mathrm{red}}\{u_0 > 0\}) \subset \bigcup_{i=1}^{N_{\delta}} B_{t_i}(y_i)
    \]
    such that
    \[
      \sum_{i=1}^{N_{\delta}} t_i^{n-1} \leq \delta.
    \]
    Split the left hand side of condition \autoref{item:3} as follows:
    \begin{align*}
      \int \nabla u_0\cdot\nabla\eta
      = \int_{\bigcup_{i=1}^{N_{\delta}} B_{t_i}(y_i)} \nabla u_0\cdot\nabla\eta +
      \int_{\{u > 0\} \setminus \bigcup_{i=1}^{N_{\delta}} B_{t_i}(y_i)} \nabla u_0\cdot\nabla\eta
      \eqqcolon I_1 + I_2.
    \end{align*}
    For the first integral we have
    \[
      |I_1| \leq C_1\sum_{i=1}^{N_{\delta}} t_i^n \leq C_1\delta^{\frac{n}{n-1}}
    \]
    and therefore \(I_1 \to 0\) as \(\delta \to 0\). For the second integral,
    \(\partial_{\mathrm{red}}\{u_0 > 0\}\) being a \(C^{1,\alpha}\)-surface implies that
    \(u_0 \in C^{1,\alpha}(\overline{\{u_0 > 0\}} \cap (\supp \eta \setminus \bigcup_{i=1}^{N_{\delta}}
    B_{t_i}(y_i)))\) by classical Schauder regularity estimates. Also, by
    \cite[Theorem 4.4]{MR3667700} we know \(\partial_{\mathrm{red}}\{u_{r_k} > 0\}\) is a
    \(C^{1,\alpha}\)-surface for every \(k\) as well. In particular, \(u_{r_k}\) is a weak
    solution for \(f_{r_k} = r_kf(x_0 + r_kx)\) and
    \(Q_{r_k}(x) \coloneqq Q(x_0 + r_kx)\) satisfying
    \(\partial_{\nu}u_{r_k} = Q_{r_k}\) on \(\partial_{\mathrm{red}}\{u_{r_k} > 0\}\) in the
    classical sense by Remark \ref{rem:3}. Moreover, the smoothness of \(Q\) implies
    that \(Q_{r_k} \to Q(x_0)\) locally uniformly. Thus, since
    \(\partial_{\mathrm{red}}\{u_{r_k} > 0\}\) are \(C^{1,\alpha}\)-surfaces converging to
    \(\partial_{\mathrm{red}}\{u_0 > 0\}\) in the \(C^{1,\alpha}\) norm, we deduce that
    \(\partial_{\nu}u_0 = Q(x_0)\) on \(\partial_{\mathrm{red}}\{u_0 > 0\}\) in the classical sense.
    Integrating by parts and using that \(\Delta u_0 = 0\) in \(\{u_0 > 0\}\) we obtain
    \begin{align*}
      I_2 &= -\int_{\{u_0 > 0\} \cap \partial(\bigcup_{i=1}^{N_{\delta}} B_{t_i}(y_i))} \eta\partial_{\nu}u_0 d\cH^{n-1} -
            \int_{\partial\{u_0 > 0\} \setminus \bigcup_{i=1}^{N_{\delta}} B_{t_i}(y_i)} \eta\partial_{\nu}u_0 d\cH^{n-1} \\
          &= -\int_{\{u_0 > 0\} \cap \partial(\bigcup_{i=1}^{N_{\delta}} B_{t_i}(y_i))} \eta\partial_{\nu}u_0 d\cH^{n-1} -
            \int_{\partial_{\mathrm{red}}\{u_0 > 0\} \setminus \bigcup_{i=1}^{N_{\delta}} B_{t_i}(y_i)} Q(x_0)\eta d\cH^{n-1}.
    \end{align*}
    Since
    \[
      \left| \int_{\{u_0 > 0\} \cap \partial(\bigcup_{i=1}^{N_{\delta}} B_{t_i}(y_i))} \eta\partial_{\nu}u_0 d\cH^{n-1} \right| \leq
      C_2\sum_{i=1}^{N_{\delta}} t_i^{n-1} \leq C_2\delta,
    \]
    letting \(\delta \to 0\) we see that
    \[
      I_2 = -\int_{\partial_{\mathrm{red}}\{u_0 > 0\}} Q(x_0)\eta d\cH^{n-1},
    \]
    and therefore
    \[
      \int \nabla u_0\cdot\nabla\eta = -\int_{\partial_{\mathrm{red}}\{u_0 > 0\}} Q(x_0)\eta d\cH^{n-1}.
    \]
    Thus, \(u_0\) satisfies condition \autoref{item:3} with \(f \equiv 0\) and
    \(Q \equiv Q(x_0)\) and the proof is complete.\qedhere
  \end{enumerate}
\end{proof}
The uniform density condition stated in Proposition \ref{prop:5} does not hold in
general for all weak solutions. Consider, for instance, the following example: let
\[
  u(x) \coloneqq
  \begin{cases}
    \frac{1}{2}(\frac{1}{2} - |x - \frac{1}{2}e_n|^2) \quad&\text{if } x \in B_{1/2}(\frac{1}{2}e_n), \\
    \frac{1}{2}(\frac{1}{2} - |x + \frac{1}{2}e_n|^2) \quad&\text{if } x \in B_{1/2}(-\frac{1}{2}e_n), \\
    0 \quad&\text{otherwise}.
  \end{cases}
\]
Then \(u\) is a weak solution in \(B_1\) for \(f \equiv n\), \(Q \equiv 1\) (in fact,
\(u\) is a classical solution outside the origin) and
\[
  \partial\{u > 0\} = \{0\} \cup \partial_{\mathrm{red}}\{u > 0\}.
\]
However, the contact set \(\{u = 0\}\) has density zero at the origin. Thus,
despite being a weak solution, \(u\) fails to satisfy \eqref{eq:densityEst}.

This example also shows that even for blow-up sequences, a very specific kind
of sequence, the blow-up limit need not be a weak solution. Indeed, if
\(u_{r_k}\) is a blow-up sequence of \(u\) at \(0\), then the blow-up limit
\(u_0\) is
\[
  u_0(x) = |x\cdot e_n|.
\]
In particular,
\[
  \{u_0 = 0\} = \{x_n = 0\},
\]
and consequently \(|\{u_0 = 0\}| = 0\). Thus, we infer that \(u_0\) cannot be
a weak solution (see Remark \ref{rem:5}).

For our purposes, however, the Lipschitz regularity of the free boundary
guarantees the density estimate required by Proposition \ref{prop:5}. Hence,
we have the following corollary.
\begin{corollary}
  \label{cor:3}
  Let \(u\) be a weak solution in \(\Omega\) such that \(\partial\{u > 0\}\) is Lipschitz.
  For any point \(x_0 \in \partial\{u > 0\} \cap \Omega\) and any blow-up sequence
  \(u_{r_k}\) at \(x_0\), there exists a weak solution
  \(u_0 \in C^{0,1}_{\mathrm{loc}}(\R^n)\) with \(f \equiv 0\) and
  \(Q \equiv Q(x_0)\) such that, up to subsequence, \(u_{r_k} \to u_0\) in the sense
  of Lemma \ref{lem:1}.
\end{corollary}
\begin{proof}
  The result is an immediate consequence of Proposition \ref{prop:5} and the fact
  that Lipschitz boundaries satisfy \eqref{eq:densityEst}.
\end{proof}

\section{Homogeneity of blow-up limits}
\label{sec:homogeneity-blow-ups}

In this section we prove the 1-homogeneity of blow-up limits by means of a Weiss
monotonicity formula.

\subsection{Stationary solutions}

To prove the monotonicity formula we first show the important fact that weak
solutions are also stationary solutions.
\begin{definition}
  \label{defn:2}
  Let \(u \in H^1_{\mathrm{loc}}(B_1)\) be a non-negative function such that
  \(u \in C(B_1) \cap C^2(\{u > 0\})\). We say \(u\) is a stationary solution of
  \eqref{eq:inhom-alt-caff-free-bound-form} if the inner variation of the functional
  \[
    \mathcal{F}_{f,Q}(v) = \frac{1}{2}\int_{B_1}\left(|\nabla v|^2 + Q^2\chi_{\{v > 0\}}\right)dx - \int_{B_1} fv dx
  \]
  vanishes at \(u\), that is, for any \(\xi \in C^{\infty}_{\mathrm{c}}(B_1;\R^n)\) we have
  \begin{align*}
    0 &= \delta\mathcal{F}_{f,Q}(u_{\varepsilon},B_1)[\xi]
        = \frac{1}{2}\int_{B_1}\left(|\nabla u|^2\Div \xi - 2\nabla u \cdot D\xi\nabla u + \Div(Q^2\xi)\chi_{\{u > 0\}}\right)dx
        - \int_{B_1} \Div(f\xi)u dx,
  \end{align*}
  where \(u_{\varepsilon} = u \circ \tau_{\varepsilon}^{-1}\) and
  \(\tau_{\varepsilon}\) is the diffeomorphism defined by
  \(\tau_{\varepsilon}(x) = x + \varepsilon\xi(x)\) for \(x \in B_1\) and \(\varepsilon\) small.
\end{definition}
\begin{proposition}
  \label{prop:6}
  Every weak solution is a stationary solution.
\end{proposition}
\begin{proof}
  We follow the ideas of the proof of \cite[Theorem 5.1]{MR1620644}. Let \(u\) be a
  weak solution in \(B_1\). We will use the facts that
  \(\cH^{n-1}(\partial\{u > 0\} \setminus \partial_{\mathrm{red}}\{u > 0\}) = 0\), that
  \(\partial_{\mathrm{red}}\{u > 0\}\) is open relative to \(\partial\{u > 0\}\) and that
  \(\partial_{\mathrm{red}}\{u > 0\}\) is locally a \(C^{\infty}\)-surface in \(B_1\) (see
  \cite[Corollary 4.1]{MR3667700}).

  Let \(\xi \in C^{\infty}_{\mathrm{c}}(B_1;\R^n)\). For any
  \(\delta > 0\) we can find a covering
  \[
    \supp \xi \cap (\partial\{u > 0\} \setminus \partial_{\mathrm{red}}\{u > 0\}) \subset \bigcup_{i=1}^{\infty} B_{r_i}(x_i)
  \]
  such that \(x_i \in \partial\{u > 0\} \setminus \partial_{\mathrm{red}}\{u > 0\}\) and \(\sum_{i=1}^{\infty} r_i^{n-1} \leq \delta\). Since \(\partial\{u > 0\} \setminus \partial_{\mathrm{red}}\{u > 0\}\) is closed,
  \(\supp \xi \cap (\partial\{u > 0\} \setminus \partial_{\mathrm{red}}\{u > 0\})\) is compact and we may
  pass to a finite subcovering
  \[
    \supp \xi \cap (\partial\{u > 0\} \setminus \partial_{\mathrm{red}}\{u > 0\}) \subset \bigcup_{i=1}^{N_{\delta}} B_{r_i}(x_i).
  \]
  Denote \(B_i \coloneqq B_{r_i}(x_i)\) and split the inner variation as follows:
  \begin{align*}
    \int_{B_1} &\left(\frac{1}{2}|\nabla u|^2\Div \xi - \nabla u \cdot D\xi \nabla u + \frac{1}{2}\Div(Q^2\xi)\chi_{\{u > 0\}} - \Div(f\xi)u\right) \\
          &= \int_{\{u > 0\} \cap \bigcup_{i=1}^{N_{\delta}} B_i} \left(\frac{1}{2}|\nabla u|^2\Div \xi - \nabla u \cdot D\xi \nabla u + \frac{1}{2}\Div(Q^2\xi) - \Div(f\xi)u\right) \\
          &\quad+ \int_{\{u > 0\} \setminus \bigcup_{i=1}^{N_{\delta}} B_i} \left(\frac{1}{2}|\nabla u|^2\Div \xi - \nabla u \cdot D\xi \nabla u + \frac{1}{2}\Div(Q^2\xi) - \Div(f\xi)u\right) \\
          &\eqqcolon I_1 + I_2.
  \end{align*}

  For the first integral we have that
  \begin{equation}
    \label{eq:5}
    |I_1| \leq C_1\sum_{i=1}^{N_{\delta}} r_i^n \leq C_1\delta^{\frac{n}{n-1}}
  \end{equation}
  using the fact that
  \(r_k^{n-1} \leq \sum_{i=1}^{N_{\delta}} r_i^{n-1} \leq \delta\) for all \(k\). On the other hand, by
  Remark \ref{rem:3} we also know
  \(u \in C^{\infty}(\overline{\{u > 0\}} \cap (\supp \xi \setminus \bigcup_{i=1}^{N_{\delta}} B_{r_i}(x_i)))\) and
  the fact that \(u\) fulfills the boundary condition
  \(|\nabla u| = \partial_{\nu}u = Q\) on
  \(\partial\{u > 0\} \cap (\supp \xi \setminus \bigcup_{i=1}^{N_{\delta}} B_{r_i}(x_i))\) in the classical sense.
  Moreover, in \(\{u > 0\}\) we have
  \begin{align*}
    \frac{1}{2}|\nabla u|^2\Div \xi &- \nabla u \cdot D\xi\nabla u = \\
                   &= \sum_{i,j} \frac{1}{2}(\partial_iu)^2 \partial_j\xi_j - \partial_ju \partial_i\xi_j \partial_iu \\
                   &= \sum_{i,j} \frac{1}{2}(\partial_iu)^2 \partial_j\xi_j - \partial_i(\partial_ju\,\xi_j\,\partial_iu) + \partial_{ij}u\,\xi_j\,\partial_iu + \partial_ju\,\xi_j\,\partial_{ii}u \\
                   &= \sum_{i,j} \frac{1}{2}(\partial_iu)^2 \partial_j\xi_j - \partial_i(\partial_ju\,\xi_j\,\partial_iu) + \frac{1}{2}\partial_j(\partial_iu\,\xi_j\,\partial_iu) - \frac{1}{2}(\partial_iu)^2 \partial_j\xi_j + \partial_ju\,\xi_j\,\partial_{ii}u \\
                   &= \sum_{i,j} -\partial_i(\partial_ju \xi_j \partial_iu) + \frac{1}{2}\partial_j(\partial_iu\,\xi_j\,\partial_iu) + \partial_ju\,\xi_j\,\partial_{ii}u \\
                   &= \Div\left(\frac{1}{2}|\nabla u|^2\xi - (\nabla u\cdot\xi)\nabla u\right) - f\nabla u\cdot\xi,
  \end{align*}
  and therefore
  \begin{align*}
    \frac{1}{2}|\nabla u|^2\Div \xi &- \nabla u \cdot D\xi\nabla u - \Div(f\xi)u = \\
                   &= \Div\left(\frac{1}{2}|\nabla u|^2\xi - (\nabla u\cdot\xi)\nabla u\right)
                     - f\nabla u\cdot\xi - u\Div(f\xi) \\
                   &= \Div\left(\frac{1}{2}|\nabla u|^2\xi - (\nabla u\cdot\xi)\nabla u\right) - \Div(uf\xi).
  \end{align*}
  Thus, since
  \(\supp \xi \cap \left(\{u > 0\} \setminus \bigcup_{i=1}^{N_{\delta}} B_i\right)\) is a piecewise
  \(C^1\) domain, we can apply the divergence theorem to \(I_2\) to obtain
  \begin{align*}
    I_2 &= -\int_{\partial\left(\{u > 0\} \setminus \bigcup_{i=1}^{N_{\delta}} B_i\right)}
          \left(\frac{1}{2}|\nabla u|^2\xi\cdot\nu - \nabla u\cdot\xi\nabla u\cdot\nu + \frac{1}{2}Q^2\xi\cdot\nu - uf\xi\cdot\nu\right)d\cH^{n-1} \\
        &= -\int_{\partial\{u > 0\} \setminus \bigcup_{i=1}^{N_{\delta}} B_i}
          \frac{1}{2}\left(Q^2 - |\nabla u|^2\right)\xi\cdot\nu d\cH^{n-1} \\
        &\quad- \int_{\{u > 0\} \cap \partial\left(\bigcup_{i=1}^{N_{\delta}} B_i\right)}
          \left(\frac{1}{2}|\nabla u|^2\xi\cdot\nu - \nabla u\cdot\xi\nabla u\cdot\nu + \frac{1}{2}Q^2\xi\cdot\nu - uf\xi\cdot\nu\right)d\cH^{n-1} \\
        &= -\int_{\{u > 0\} \cap \partial\left(\bigcup_{i=1}^{N_{\delta}} B_i\right)}
          \left(\frac{1}{2}|\nabla u|^2\xi\cdot\nu - \nabla u\cdot\xi\nabla u\cdot\nu + \frac{1}{2}Q^2\xi\cdot\nu - uf\xi\cdot\nu\right)d\cH^{n-1}.
  \end{align*}
  Hence, we see that
  \begin{equation}
    \label{eq:6}
    |I_2| \leq C_2\sum_{i=1}^{N_{\delta}} r_i^{n-1} \leq C_2\delta.
  \end{equation}
  Combining \eqref{eq:5} and \eqref{eq:6} yields
  \[
    \left|\int_{B_1} \left(\frac{1}{2}|\nabla u|^2\Div \xi - \nabla u \cdot D\xi\nabla u + \frac{1}{2}\Div(Q^2\xi)\chi_{\{u > 0\}} + \Div(f\xi)u\right)\right| \leq
    C_1\delta^{\frac{n}{n-1}} + C_2\delta
  \]
  and letting \(\delta \to 0\) gives the desired result.
\end{proof}

\subsection{Weiss monotonicity formula}

For a given boundary data \(Q\) and a non-negative function \(u \in H^1(B_1)\), we
define the modified Weiss function
\[
  W_Q(u) \coloneqq
  \int_{B_1} |\nabla u|^2dx + Q(0)^2|\{u > 0\} \cap B_1| - \int_{\partial B_1} u^2d\cH^{n-1}.
\]
Then, for any \(u \in H^1(B_r(x_0))\), \(u \geq 0\), the scaling
\[
  u_r(x) = \frac{1}{r}u(x_0 + rx), \qquad Q_r(x) = Q(x_0 + rx),
\]
satisfies that \(u_r \in H^1(B_1)\) and
\begin{align*}
  W_Q(u_r) &= \int_{B_1} |\nabla u_r|^2dx
             + Q_r(0)^2|\{u_r > 0\} \cap B_1|
             - \int_{\partial B_1} u_r^2d\cH^{n-1} \\
           &= \frac{1}{r^n}\int_{B_r(x_0)} |\nabla u|^2dx
             + \frac{Q(x_0)^2}{r^n}|\{u > 0\} \cap B_r(x_0)|
             - \frac{1}{r^{n+1}}\int_{\partial B_r(x_0)} u^2d\cH^{n-1}.
\end{align*}

By \cite[Lemmas 9.1, Lemma 9.2]{MR4807210} the function \(r \mapsto W_Q(u_r)\) is
continuous and a.e. differentiable for any \(u \in H^1(B_{\delta}(x_0))\) and
\(r \in (0,\delta)\). Moreover, if \(z_r(x) = |x|u_r(\frac{x}{|x|})\) is the 1-homogeneous
extension of \(u_r\) from \(\partial B_1\) to \(B_1\), then \cite[Lemma 9.2]{MR4807210} also
gives an explicit expression for the derivative of \(W_Q(u_r)\), namely
\[
  \frac{\partial}{\partial r}W_Q(u_r) =
  \frac{n}{r}\left( W_Q(z_r) - W_Q(u_r) \right)
  + \frac{1}{r}\int_{\partial B_1} |\nabla u_r \cdot x - u_r|^2d\cH^{n-1}.
\]
\begin{lemma}
  \label{lem:4}
  Let \(u\) be a weak solution in \(B_1\). Then, for every \(x_0 \in B_1\) and every
  \(0 < r < r_0\) with \(r_0 \coloneqq \frac{1}{2}\dist(x_0,\partial B_1)\), we have
  \[
    W_Q(z_r) - W_Q(u_r) \geq \frac{1}{n}\int_{\partial B_1} |\nabla u_r \cdot x - u_r|^2d\cH^{n-1} - Cr,
  \]
  where \(C\) is a positive constant that depends on \(n\),
  \(\|u\|_{\infty}\), \(\|f\|_{\infty}\), \(\|\nabla f\|_{\infty}\), \(\|Q\|_{\infty}\) and \(\|\nabla Q\|_{\infty}\).
\end{lemma}
\begin{proof}
  We follow the ideas of \cite[Lemma 9.8]{MR4807210}. Without loss of generality
  assume \(x_0 = 0\). For every \(\varepsilon > 0\), consider the function
  \(\phi_{\varepsilon} \in C^{\infty}_{\mathrm{c}}(B_r)\) such that
  \[
    \phi_{\varepsilon} = 1 \quad\text{in } B_{(1-\varepsilon)r}, \qquad
    \nabla\phi_{\varepsilon}(x) = -\frac{1}{r\varepsilon}\frac{x}{|x|} + o(\varepsilon) \quad\text{in } B_r \setminus B_{(1-\varepsilon)r},
  \]
  and define the test vector field \(\xi_{\varepsilon}(x) = \phi_{\varepsilon}(x)x\). Since
  \(u\) is a stationary solution in \(B_1\) by Proposition \ref{prop:6}, testing
  \(u\) against \(\xi_{\varepsilon}\) we obtain
  \begin{align*}
    0 = \delta\mathcal{F}_{f,Q}(u)[\xi_{\varepsilon}]
    &= \frac{1}{2}\int \left( |\nabla u|^2\Div \xi_{\varepsilon} - 2\nabla u\cdot D\xi_{\varepsilon}\nabla u
      + (Q^2\Div \xi_{\varepsilon} + \nabla Q^2\cdot\xi_{\varepsilon})\chi_{\{u > 0\}}\right)dx \\
    &\quad- \int (\nabla f\cdot\xi_{\varepsilon} + f\Div \xi_{\varepsilon})u dx \\
    &= \frac{1}{2}\int \left( - 2\phi_{\varepsilon}|\nabla u|^2 - 2(\nabla u\cdot x)(\nabla u\cdot\nabla\phi_{\varepsilon})\right)dx \\
    &\quad+ \frac{1}{2}\int \left((n\phi_{\varepsilon} + \nabla\phi_{\varepsilon}\cdot x)(|\nabla u|^2 + Q^2\chi_{\{u > 0\}})
      + \phi_{\varepsilon}\nabla Q^2\cdot x\chi_{\{u > 0\}}\right)dx \\
    &\quad- \int \left( (n\phi_{\varepsilon} + \nabla\phi_{\varepsilon}\cdot x)uf + \phi_{\varepsilon}u\nabla f\cdot x \right)dx \\
    &= \frac{1}{2}\int_{B_r} \left( (n-2)|\nabla u|^2 + (nQ^2 + \nabla Q^2\cdot x)\chi_{\{u > 0\}} \right)\phi_{\varepsilon}dx \\
    &\quad+ \frac{1}{2r\varepsilon}\int_{B_r \setminus B_{(1-\varepsilon)r}} \left( 2\left(\nabla u\cdot \frac{x}{|x|}\right)^2
      - |\nabla u|^2 - Q^2\chi_{\{u > 0\}} \right)|x|dx \\
    &\quad- \int_{B_r} \left( nuf + u\nabla f\cdot x \right)\phi_{\varepsilon}dx + \frac{1}{r\varepsilon}\int_{B_r \setminus B_{(1-\varepsilon)r}} uf|x|dx + o(\varepsilon).
  \end{align*}
  Now we can let \(\varepsilon \to 0\) and multiply everything by 2 to get
  \begin{align*}
    0 &= \int_{B_r} \left( (n-2)|\nabla u|^2 + (nQ^2 + \nabla Q^2\cdot x)\chi_{\{u > 0\}} \right)dx \\
      &\quad+ r\int_{\partial B_r} \left( 2\left(\nabla u\cdot \frac{x}{|x|}\right)^2
        - |\nabla u|^2 - Q^2\chi_{\{u > 0\}} \right)d\cH^{n-1} \\
      &\quad- 2\int_{B_r} \left( nuf + u\nabla f\cdot x \right)dx + 2r\int_{\partial B_r} uf d\cH^{n-1}.
  \end{align*}
  Using that \(-\Delta u = f\) in \(\{u > 0\}\) we can compute
  \[
    \int_{B_r} |\nabla u|^2dx
    = \int_{B_r} \Div(u\nabla u)dx + \int_{B_r} ufdx
    = \int_{\partial B_r} u\nabla u\cdot \frac{x}{|x|} d\cH^{n-1} + \int_{B_r} ufdx.
  \]
  Inserting this computation in the term \(-2|\nabla u|^2\) of the first integral
  and rescaling everything to \(B_1\) we find that
  \begin{equation}
    \label{eq:weiss_computation}
    \begin{aligned}
      0 &= \int_{B_1} \left( n|\nabla u_r|^2 + (nQ_r^2 + \nabla Q_r^2\cdot x)\chi_{\{u_r > 0\}} \right)dx \\
        &\quad+ \int_{\partial B_1} \left( - 2u_r\nabla u_r\cdot x + 2(\nabla u_r\cdot x)^2
          - |\nabla u_r|^2 - Q_r^2\chi_{\{u_r > 0\}} \right)d\cH^{n-1} \\
        &\quad- 2\int_{B_1} \left( (n+1)u_rf_r + u_r\nabla f_r\cdot x \right)dx + 2\int_{\partial B_1} u_rf_r d\cH^{n-1}.
    \end{aligned}
  \end{equation}
  With this we see that
  \begin{align}
    \label{eq:weiss_error}
    W_Q(z_r) &- W_Q(u_r) = \nonumber \\
             &= \int_{B_1} |\nabla z_r|^2dx + Q_r(0)^2|\{z_r > 0\} \cap B_1|
               - \int_{B_1} |\nabla u_r|^2dx - Q_r(0)^2|\{u_r > 0\} \cap B_1| \nonumber \\
             &= \frac{1}{n}\int_{\partial B_1} \left( |\nabla u_r|^2 - |\nabla u_r\cdot x|^2 \right)d\cH^{n-1}
               + \frac{1}{n}\int_{\partial B_1} u_r^2d\cH^{n-1} \nonumber \\
             &\quad+ \frac{1}{n}Q_r(0)^2\cH^{n-1}(\{u_r > 0\} \cap \partial B_1)
               - \int_{B_1} |\nabla u_r|^2dx - Q_r(0)^2|\{u_r > 0\} \cap B_1| \nonumber \\
             &= \frac{1}{n}\int_{\partial B_1} |\nabla u_r\cdot x - u_r|^2d\cH^{n-1}
               + \frac{1}{n}\int_{\partial B_1} \left( |\nabla u_r|^2 - 2|\nabla u_r\cdot x|^2 \right)d\cH^{n-1} \nonumber \\
             &\quad+ \frac{2}{n}\int_{\partial B_1} u_r\nabla u_r\cdot xd\cH^{n-1}
               + \frac{1}{n}Q_r(0)^2\cH^{n-1}(\{u_r > 0\} \cap \partial B_1) \nonumber \\
             &\quad- \int_{B_1} |\nabla u_r|^2dx - Q_r(0)^2|\{u > 0\} \cap B_1| \nonumber \\
             &= \frac{1}{n}\int_{\partial B_1} |\nabla u_r\cdot x - u_r|^2d\cH^{n-1} + \sum_{i=1}^6 \text{Error}_i.
  \end{align}
  Here we have used: in the first step, the fact that \(u_r \equiv z_r\) on
  \(\partial B_1\); in the second step, explicit expressions for integrals involving
  \(z_r\) by integration in polar coordinates (see \cite[(9.5)]{MR4807210}); in the
  third step, we add and subtract
  \(\frac{1}{n}\int_{\partial B_1} |\nabla u_r\cdot x - u_r|^2d\cH^{n-1}\). In the last step, we have
  introduced the terms in \eqref{eq:weiss_computation} (divided by \(n\)) which
  leaves us with the following error terms:
  \begin{align*}
    \text{Error}_1 &\coloneqq \frac{1}{n}\left( Q_r(0)^2\cH^{n-1}(\{u_r > 0\} \cap \partial B_1)
                     - \int_{\partial B_1} Q_r^2\chi_{\{u_r > 0\}}d\cH^{n-1}\right), \\
    \text{Error}_2 &\coloneqq -\left( Q_r(0)^2|\{u_r > 0\} \cap B_1|
                     - \int_{B_1} Q_r^2\chi_{\{u_r > 0\}}dx\right), \\
    \text{Error}_3 &\coloneqq \frac{1}{n}\int_{B_1} \nabla Q_r^2\cdot x\chi_{\{u_r > 0\}}dx,
                     \hspace{3em} \text{Error}_4 \coloneqq -\frac{2}{n}\int_{B_1} u_r\nabla f_r\cdot xdx, \\
    \text{Error}_5 &\coloneqq -\frac{2(n+1)}{n}\int_{B_1} u_rf_rdx,
                     \hspace{3.925em} \text{Error}_6 \coloneqq \frac{2}{n}\int_{\partial B_1} u_rf_rd\cH^{n-1}.
  \end{align*}
  For the first two error terms we have
  \begin{align*}
    |\text{Error}_1|,\,|\text{Error}_2| &\leq C_n\|Q_r(0) + Q_r\|_{\infty}\|Q_r(0) - Q_r\|_{\infty} \\
                                        &\leq 2C_n\|Q\|_{\infty}\|\nabla Q\|_{\infty}r \\
                                        &= O(r).
  \end{align*}
  For the third term we use that
  \(\nabla Q_r^2(x) = 2Q_r(x)\nabla Q_r(x) = 2rQ_r(x)\nabla Q(rx)\) to obtain
  \[
    |\text{Error}_3|
    \leq C_nr\int_{B_1} |Q_r(x)\nabla Q(rx)|d\cH^{n-1}
    \leq C_n\|Q\|_{\infty}\|\nabla Q\|_{\infty}r
    = O(r).
  \]
  For the fourth term we use the boundedness of \(u\) to obtain
  \[
    |\text{Error}_4| \leq C_n\|u\|_{\infty}\|\nabla f\|_{\infty}r = O(r).
  \]
  For the last two error terms we have
  \begin{align*}
    |\text{Error}_5| &\leq C_n\|f\|_{\infty}\int_{B_1} u(rx) dx = O(r), \\
    |\text{Error}_6| &\leq C_n\|f\|_{\infty}\int_{\partial B_1} u(rx) dx = O(r),
  \end{align*}
  since \(u\) is Lipschitz in \(B_{r_0}\). Inserting this estimates in
  \eqref{eq:weiss_error} yields
  \[
    W_Q(z_r) - W_Q(u_r) \geq \frac{1}{n}\int_{\partial B_1} |\nabla u_r\cdot x - u_r|^2d\cH^{n-1} - Cr
  \]
  as claimed.
\end{proof}
\begin{proposition}
  \label{prop:7}
  Let \(u\) be a weak solution in \(B_1\) and let
  \(x_0 \in \partial\{u > 0\} \cap B_1\) be any free boundary point. For any blow-up sequence
  \(u_{r_k}\) of \(u\) at \(x_0\), the blow-up limit \(u_0\) (see Lemma \ref{lem:1})
  is a 1-homogeneous function.
\end{proposition}
\begin{proof}
  By Lemma \ref{lem:4}, for all \(0 < r < r_0\) with
  \(r_0 \coloneqq \dist(x_0,\partial\Omega)\) we have
  \begin{equation}
    \label{eq:weiss_deriv_lower_bound}
    \frac{\partial}{\partial r}W_Q(u_r)
    \geq \frac{2}{r}\int_{\partial B_1} |\nabla u_r\cdot x - u_r|^2d\cH^{n-1} - C
    \geq -C.
  \end{equation}
  Since \(u\) is Lipschitz in \(B_{r_0}(x_0)\), we also have \(|u_r(x)| \leq L|x|\) for
  all \(0 < r < r_0\) and therefore
  \[
    W_Q(u_r) \geq -\int_{\partial B_1} u_r^2d\cH^{n-1} \geq -L^2.
  \]
  Hence, \(r \mapsto W_Q(u_r)\) is a continuous and a.e. differentiable function in
  the interval \((0,r_0)\) which is bounded below and whose derivative is also
  bounded below in \((0,r_0)\). It follows that the limit
  \begin{equation}
    \label{eq:weiss_limit_exists}
    l = \lim_{r \to 0} W_Q(u_r)
  \end{equation}
  exists and is finite.

  Fix \(0 < S < R < \infty\). Integrating \eqref{eq:weiss_deriv_lower_bound} from
  \(Sr_k\) to \(Rr_k\) gives us
  \[
    W_Q(u_{Rr_k}) - W_Q(u_{Sr_k})
    = \int_{Sr_k}^{Rr_k} \frac{2}{r}\int_{\partial B_1} |\nabla u_r\cdot x - u_r|^2d\cH^{n-1}dr
    - C(Rr_k - Sr_k).
  \]
  We can rewrite the integral term as follows:
  \begin{align*}
    \int_{Sr_k}^{Rr_k} \frac{2}{r}\int_{\partial B_1} |\nabla u_r\cdot x - u_r|^2d\cH^{n-1}dr &= \int_S^R \frac{2}{r}\int_{\partial B_1} |\nabla u_{r_kr}\cdot x - u_{r_kr}|^2d\cH^{n-1}dr \\
                                                                        &= \int_S^R \frac{2}{r^n}\int_{\partial B_r} \left|\nabla u_{r_k}\cdot\frac{x}{r} - \frac{u_{r_k}}{r}\right|^2d\cH^{n-1}dr \\
                                                                        &= \int_{B_R \setminus B_S} 2|x|^{-n-2}|\nabla u_{r_k}\cdot x - u_{r_k}|^2dx.
  \end{align*}
  Inserting this back into the equation yields
  \[
    W_Q(u_{Rr_k}) - W_Q(u_{Sr_k})
    = \int_{B_R \setminus B_S} 2|x|^{-n-2}|\nabla u_{r_k}\cdot x - u_{r_k}|^2dx
    - C(Rr_k - Sr_k).
  \]
  Thus, letting \(r_k \to 0\) and using the fact that \(u_{r_k} \to u_0\) in
  \(H^1(B_R)\) we find that
  \[
    \int_{B_R \setminus B_S} 2|x|^{-n-2}|\nabla u_0\cdot x - u_0|^2dx = 0.
  \]
  It follows that \(\nabla u_0\cdot x - u_0 = 0\) a.e. in
  \(B_R \setminus B_S\). Since this holds for any \(0 < S < R < \infty\), we conclude that
  \(\nabla u_0\cdot x - u_0 = 0\) a.e. in \(\R^n\) and therefore \(u_0\) is
  1-homogeneous as claimed.
\end{proof}

\section{Proof of Theorem \ref{thm:1}}
\label{sec:lip-implies-smooth}

At this point we have all the main ingredients for the proof of Theorem \ref{thm:1}.
As already explained in the introduction, the last part of the proof can be carried
out either by classifying the blow-up limits at free boundary points, or by
leveraging the known regularity of free boundaries of viscosity solutions. However,
before we start with this last step we prove an auxiliary result concerning the
definition of weak solution in the case of Lipschitz domains.

Generally, one would expect a weak notion of solution to only ask for condition
\autoref{item:3} in Definition \ref{defn:1}. Without any additional assumptions, none
of the three conditions in Definition \ref{defn:1} can be used to deduce the other.
Nevertheless, in the case of Lipschitz domains we can prove that \autoref{item:3} is
a sufficient condition to be a weak solution.
\begin{lemma}
  \label{lem:5}
  Let \(\Omega \subset \R^n\) be a bounded Lipschitz domain such that
  \(0 \in \partial \Omega\). Assume \(u \in H^1_{\mathrm{loc}}(B_1)\) is a non-negative function such
  that \(u = 0\) in \(B_1 \setminus \Omega\) and
  \begin{equation}
    \label{eq:lip-weak-form}
    -\int_{B_1} \nabla u\cdot\nabla\eta = \int_{\partial \Omega \cap B_1} Q\eta d\cH^{n-1} - \int_{\Omega \cap B_1} f\eta,
    \qquad \forall\eta \in C^{\infty}_{\mathrm{c}}(B_1).
  \end{equation}
  Then \(u\) is a weak solution in \(B_1\) in the sense of Definition \ref{defn:1}.
  Moreover, \(\{u > 0\} = \Omega \cap B_1\).
\end{lemma}
\begin{proof}
  Taking \(\eta \in C^{\infty}_{\mathrm{c}}(\Omega \cap B_1)\) in \eqref{eq:lip-weak-form} we see that
  \[
    \int_{\Omega \cap B_1} \nabla u\cdot\nabla\eta = \int_{\Omega \cap B_1} f\eta, \qquad \forall\eta \in C^{\infty}_{\mathrm{c}}(\Omega \cap B_1).
  \]
  Hence, \(u\) solves \(-\Delta u = f\) in \(\Omega \cap B_1\) in the weak sense. By the
  smoothness of \(f\) and the fact that \(\Omega\) is Lipschitz we deduce that
  \(u \in C(\overline{\Omega \cap B_1}) \cap C^{\infty}(\Omega \cap B_1)\). In particular,
  \(u\) solves \(-\Delta u = f\) in the classical sense in \(\Omega \cap B_1\). Also, since
  \(u = 0\) on \(\partial\Omega \cap B_1\), we infer that \(u \in C(B_1)\). Now, the fact that
  \(-\Delta u = f \geq 0\) in \(\Omega \cap B_1\) implies that \(u\) is superharmonic in
  \(\Omega \cap B_1\), and by the non-negativity of \(u\) we conclude that \(u > 0\) in
  \(\Omega \cap B_1\).

  We have just proved that \(\{u > 0\} \cap B_1 = \Omega \cap B_1\) and therefore
  \(\partial\{u > 0\} \cap B_1 = \partial\Omega \cap B_1\) as well. Thus, we can rewrite
  \eqref{eq:lip-weak-form} as
  \begin{equation}
    \label{eq:lip-weak-form-rewrite}
    -\int_{B_1} \nabla u\cdot\nabla\eta = \int_{\partial\{u > 0\} \cap B_1} Q\eta d\cH^{n-1} - \int_{\{u > 0\}} f\eta,
    \qquad \forall\eta \in C^{\infty}_{\mathrm{c}}(B_1).
  \end{equation}
  Since \(\Omega\) is Lipschitz, \(\partial_{\mathrm{red}}\Omega = \partial \Omega\)
  \(\cH^{n-1}\)-a.e. and the equation above implies that \(u\) satisfies
  \autoref{item:3}. Also, \(u \in C(B_1)\), \(u \geq 0\) and \(-\Delta u = f\) in
  \(\Omega \cap B_1 = \{u > 0\} \cap B_1\). In other words, \(u\) also fulfills condition
  \autoref{item:1}.

  It remains to check that \(u\) satisfies \autoref{item:2}. Let
  \(D \subset\subset B_1\) be an open set and choose \(B_r(x) \subset D\) with
  \(x \in \partial\{u > 0\}\). By density, \eqref{eq:lip-weak-form-rewrite} holds for
  \(\eta \in H^1_0(B_1)\) as well. Define
  \[
    \zeta(y) \coloneqq \frac{1}{2n}(r^2 - |y-x|^2) \in H^1_0(B_1).
  \]
  Then
  \begin{align*}
    -\int_{B_r(x)} \nabla u\cdot\nabla\zeta &= -\int_{\partial B_r(x)} u\nabla\zeta\cdot\nu d\cH^{n-1} + \int_{B_r(x)} u\Delta\zeta \\
                       &= \frac{r}{n}\int_{\partial B_r(x)} u d\cH^{n-1} - \int_{B_r(x)} u,
  \end{align*}
  and by \eqref{eq:lip-weak-form-rewrite}
  \begin{equation}
    \label{eq:boundary-int-identity}
    \begin{aligned}
      \fint_{\partial B_r(x)} u d\cH^{n-1} &= \fint_{B_r(x)} u - \frac{1}{|B_r|}\int_{B_r(x)} \nabla u\cdot\nabla\zeta \\
                                    &= \fint_{B_r(x)} u
                                      + \frac{1}{|B_r|}\int_{\partial\{u > 0\} \cap B_r(x)} Q\zeta d\cH^{n-1}
                                      - \frac{1}{|B_r|}\int_{\{u > 0\} \cap B_r(x)} f\zeta.
    \end{aligned}
  \end{equation}
  Observe that since \(\partial\{u > 0\} = \partial \Omega \cap B_1\) is Lipschitz,
  \begin{equation}
    \label{eq:lipschitz-density}
    cr^{n-1} \leq \cH^{n-1}(\partial\{u > 0\} \cap B_r(x)) \leq Cr^{n-1}
  \end{equation}
  for some constants \(c,C > 0\) depending on \(D\). Thus,
  \eqref{eq:boundary-int-identity} yields
  \begin{align*}
    \fint_{\partial B_r(x)} u d\cH^{n-1} &\geq \frac{Q_{\mathrm{min}}}{|B_r|}\int_{\partial\{u > 0\} \cap B_{r/2}(x)} \zeta d\cH^{n-1}
                                    - \frac{\|f\|_{L^{\infty}(D)}}{|B_r|}\int_{\{u > 0\} \cap B_r(x)} \zeta \\
                                  &\geq c_nQ_{\mathrm{min}}\frac{r^2}{|B_r|}\cH^{n-1}(\partial\{u > 0\} \cap B_{r/2}(x))
                                    - c_n\|f\|_{\infty}r^2 \\
                                  &\geq cr - c_n\|f\|_{\infty}r^2.
  \end{align*}
  It follows that if \(r\) is smaller than some constant \(c_0\) depending on \(n\),
  \(Q_{\mathrm{min}}\) and \(\|f\|_{L^{\infty}(D)}\), then
  \[
    \fint_{\partial B_r(x)} u d\cH^{n-1} \geq cr.
  \]
  Hence, for \(r \leq c_0\) the lower bound in \autoref{item:2} holds. Given that for
  \(r \geq c_0\),
  \[
    \frac{1}{r}\fint_{\partial B_r(x)} u d\cH^{n-1}
    \geq \left( \frac{c_0}{r} \right)^{n+1}\frac{1}{c_0}\fint_{\partial B_{c_0}(x)} u
    d\cH^{n-1}
    \geq \left( \frac{c_0}{\diam(D)} \right)^{n+1}c,
  \]
  we conclude that \(u\) satisfies the lower bound in \autoref{item:2} for any ball
  \(B_r(x) \subset D\).

  Lastly, let us prove the upper bound of condition \autoref{item:2}. First, since
  \(\{u > 0\} = \Omega \cap B_1\) and \(\Omega\) is a Lipschitz domain, by Schauder estimates
  \(u \in C^{0,\alpha}(\overline{\{u > 0\}})\) for some small \(0 < \alpha < 1\). Therefore
  \(u \in C^{0,\alpha}(B_1)\). Now define
  \[
    \phi(r) \coloneqq \frac{1}{r^{\alpha}}\fint_{\partial B_r(x)} u d\cH^{n-1}.
  \]
  By the H\"older continuity of \(u\) we know that \(\phi\) is bounded. Moreover, we can
  rewrite \eqref{eq:boundary-int-identity} in terms of \(\phi\) as
  \begin{align*}
    \phi(r) &= \frac{n}{r^{n+\sigma}}\int_0^r s^{n+\alpha-1}\phi(s)ds
           + \frac{1}{r^{\alpha}|B_r|}\int_{\partial\{u > 0\} \cap B_r(x)} Q\zeta d\cH^{n-1}
           - \frac{1}{r^{\alpha}|B_r|}\int_{\{u > 0\} \cap B_r(x)} f\zeta \\
         &\leq \frac{n}{n+\alpha}\sup_{(0,r)} \phi
           + C_nQ_{\mathrm{max}}\frac{r^2}{r^{\alpha}|B_r|}\cH^{n-1}(\partial\{u > 0\} \cap B_r(x))
           + C_n\|f\|_{L^{\infty}(D)}\frac{r^2}{r^{\alpha}} \\
         &\leq \frac{n}{n+\alpha}\sup_{(0,r)} \phi + Cr^{1-\alpha},
  \end{align*}
  where we have used \eqref{eq:lipschitz-density} to bound the second term, and the
  fact that \(r \leq \diam(D)\) to estimate the last term. Consequently,
  \[
    \left( 1 - \frac{n}{n+\alpha} \right)\sup_{(0,r)} \phi \leq Cr^{1-\alpha},
  \]
  or equivalently,
  \[
    \sup_{(0,r)} \phi \leq C\frac{n+\alpha}{\alpha}r^{1-\alpha}.
  \]
  Hence,
  \[
    \fint_{\partial B_r(x)} u d\cH^{n-1} \leq Cr
  \]
  and \(u\) fulfills condition \autoref{item:2}.
\end{proof}

\subsection{Lipschitz cones}

The aim of this section is to prove Theorem \ref{thm:1} by means of classifying the
blow-up limits at free boundary points of a weak solution with Lipschitz free
boundary. Observe that by Corollary \ref{cor:3} such blow-ups are weak solutions of
the classical one-phase problem with constant boundary condition. Therefore,
throughout this subsection we will work in the particular case \(f \equiv 0\),
\(Q \equiv \text{const}\).
\begin{remark}
  \label{rem:6}
  By \cite[Remark 4.2]{MR618549}, if a function \(u\) satisfies condition
  \autoref{item:1} of Definition \ref{defn:1} with \(f \equiv 0\), then its Laplacian is
  non-negative in the distributional sense. Thus, \(\lambda = \Delta u\) is a (positive) Radon
  measure such that
  \[
    -\int_{\Omega} \nabla u\cdot\nabla\varphi dx = \int_{\Omega} \varphi d\lambda, \quad \forall\varphi \in C^{\infty}_{\mathrm{c}}(\Omega).
  \]
  Moreover, by \cite[Theorem 4.3]{MR618549} the following two conditions are
  equivalent:
  \begin{itemize}
  \item For any open set \(D \subset\subset B_1\) there exist constants
    \(0 < c_D \le C_D < \infty\) such that for any ball \(B_r(x) \subset D\) with
    \(x \in \partial\{u > 0\}\) we have
    \[
      c_D \le \frac{1}{r}\fint_{\partial B_r(x)} u d\cH^{n-1} \le C_D.
    \]
  \item For any open set \(D \subset\subset B_1\) there exist constants
    \(0 < c_D \le C_D < \infty\) such that for any ball \(B_r(x) \subset D\) with
    \(x \in \partial\{u > 0\}\) we have
    \[
      c_Dr^{n-1} \le \int_{B_r(x)} d\lambda \le C_Dr^{n-1}.
    \]
  \end{itemize}
\end{remark}
\begin{proposition}
  \label{prop:8}
  Let \(Q_0 > 0\) be a positive constant and let \(u:\R^n \rightarrow \R\) be a function which
  is constant in the \(e_n\) direction, that is,
  \[
    u(x_1,\ldots,x_n) = v(x_1,\ldots,x_{n-1}).
  \]
  Then \(u\) is a weak solution in \(\R^n\) for \(f \equiv 0\) and \(Q \equiv Q_0\) if and only
  if \(v\) is a weak solution in \(\R^{n-1}\) for \(f \equiv 0\) and \(Q \equiv Q_0\).
\end{proposition}
\begin{proof}
  We use the notation \(x = (x',x_n)\) for points in \(\R^n\) and \(B_r'\) for balls
  in \(\R^{n-1}\). In this notation we have that \(u(x) = v(x')\) and
  \(\nabla u(x) = (\nabla v(x'),0)\). Since \(u(x) = v(x')\), it is immediate to see that
  condition \autoref{item:1} holds for \(u\) if and only if it holds for \(v\).

  Let us now check condition \autoref{item:2}. If \(u\) is a weak solution, then both
  \(u\) and \(v\) satisfy \autoref{item:1}, and by Remark \ref{rem:6} it is enough to
  prove that \(v\) satisfies the following: for every bounded open set
  \(D' \subset \R^{n-1}\) and every ball \(B_r'(x') \subset D'\) with \(x' \in \partial\{v > 0\}\),
  \[
    c_{D'}r^{n-2} \le \int_{B_r'(x')} d\lambda' \le C_{D'}r^{n-2},
  \]
  where \(\lambda'\) is the Radon measure given by the distributional Laplacian of
  \(v\). Observe that for any \(\varphi \in C^{\infty}_{\mathrm{c}}(\R^n)\),
  \begin{equation}
    \label{eq:7}
    \begin{split}
      \int_{\R^n} \nabla u(x)\cdot\nabla\varphi(x)dx &= \sum_{i=1}^n\int_{\R^n} \partial_iv(x')\partial_i\varphi(x)dx \\
                              &= \sum_{i=1}^n\int_{\R^{n-1}} \partial_iv(x') \int_{\R} \partial_i\varphi(x)dx_ndx' \\
                              &= \sum_{i=1}^n\int_{\R^{n-1}} \partial_iv(x') \partial_i\left( \int_{\R} \varphi(x',x_n)dx_n \right)dx' \\
                              &= \int_{\R^{n-1}} \nabla v(x')\cdot\nabla\psi(x')dx'
    \end{split}
  \end{equation}
  where \(\psi(x') \coloneqq \int_{\R} \varphi(x',x_n)dx_n\). Therefore, since
  \(\psi \in C^{\infty}_{\mathrm{c}}(\R^{n-1})\),
  \[
    \int_{\R^n} \varphi d\lambda =
    -\int_{\R^n} \nabla u\cdot\nabla\varphi dx =
    -\int_{\R^{n-1}} \nabla v\cdot\nabla \psi dx' =
    \int_{\R^{n-1}} \psi d\lambda'
  \]
  and we see that
  \[
    \int_{\R^n} \varphi d\lambda =
    \int_{\R^{n-1}} \psi d\lambda' =
    \int_{\R^{n-1}}\int_{\R} \varphi(x',x_n) dx_nd\lambda'(x')
  \]
  for every \(\varphi \in C^{\infty}_{\mathrm{c}}(\R^n)\). Now, given any ball
  \(B_r(x) \subset \R^n\) we can take \(\varphi\) such that
  \(\chi_{B_r(x)} \le \varphi \le \chi_{B_{2r}(x)}\) and obtain that
  \begin{align*}
    \int_{B_r(x)} d\lambda &\le \int_{\R^n} \varphi d\lambda = \int_{\R^{n-1}}\int_{\R} \varphi dy_nd\lambda'(y') \le \int_{\R^{n-1}}\int_{\R} \chi_{B_{2r}(x)} dy_nd\lambda'(y') \\
                  &\le \int_{\R^{n-1}}\int_{x_n-2r}^{x_n+2r} \chi_{B_{2r}'(x')} dy_nd\lambda'(y') = 4r\int_{B_{2r}'(x')} d\lambda'.
  \end{align*}
  Similarly, taking \(\varphi\) such that \(\chi_{B_{r/2}(x)} \le \varphi \le \chi_{B_r(x)}\) we get
  \[
    r\int_{B_{r/2}'(x')} d\lambda' \le \int_{B_r(x)} d\lambda.
  \]
  Combining the two estimates we see that for any ball \(B_r(x) \subset \R^n\),
  \begin{equation}
    \label{eq:8}
    r\int_{B_{r/2}'(x')} d\lambda' \le \int_{B_r(x)} d\lambda \le 4r\int_{B_{2r}'(x')} d\lambda'.
  \end{equation}

  Let \(D' \subset \R^{n-1}\) be a bounded open set and let \(D \subset \R^n\) be a
  bounded open set large enough so that \(B_{r/2}((x',0)) \subset D\) and
  \(B_{2r}((x',0)) \subset D\) for all \(B_r'(x') \subset D'\). Since \(u\) is a weak
  solution by hypothesis, there exist constants \(c_D,C_D\) such that
  \begin{equation}
    \label{eq:9}
    c_Dr^{n-1} \le \int_{B_r(x)} d\lambda \le C_Dr^{n-1},
  \end{equation}
  for any \(B_r(x) \subset D\) with \(x \in \partial\{u > 0\}\). Hence, combining \eqref{eq:8} and
  \eqref{eq:9} we conclude that
  \[
    c_{D'}r^{n-1} \le \int_{B_r'(x')} d\lambda' \le C_{D'}r^{n-1}
  \]
  for any \(B_r'(x') \subset D'\) with \(x' \in \partial\{v > 0\}\), that is, \(v\) satisfies
  condition \autoref{item:2}. Conversely, if \(v\) is a weak solution, then once
  again both \(u\) and \(v\) fulfill \autoref{item:1}. Therefore we can repeat the
  previous argument (with some obvious modifications to the domains \(D\) and \(D'\))
  to conclude that \(u\) satisfies \autoref{item:2}.

  Lastly, let us check condition \autoref{item:3}. Assume \(u\) is a weak solution in
  \(\R^n\). Given \(\psi \in C^{\infty}_{\mathrm{c}}(\R^{n-1})\), fix a function
  \(\eta \in C^{\infty}_{\mathrm{c}}(\R)\) such that \(\int \eta = 1\) and define
  \[
    \varphi(x) \coloneqq \psi(x')\eta(x_{n+1}) \in C^{\infty}_{\mathrm{c}}(\R^n).
  \]
  Then \(\nabla\varphi(x) = (\nabla\psi(x')\eta(x_n),\psi(x')\eta'(x_n))\) and therefore
  \begin{align*}
    \int_{\R^n} \nabla u(x)\cdot\nabla\varphi(x)dx &= \int_{\R^n} \eta(x_n)\nabla v(x')\cdot\nabla\psi(x')dx \\
                            &= \int_{\R^{n-1}} \nabla v(x')\cdot\nabla\psi(x') \int_{\R} \eta(x_n)dx_ndx' \\
                            &= \int_{\R^{n-1}} \nabla v(x')\cdot\nabla\psi(x')dx'.
  \end{align*}
  Since
  \(\partial_{\mathrm{red}}\{u > 0\} \cap \{x_n = t\} = \partial_{\mathrm{red}}\{v > 0\} \times
  \{t\}\) for all \(t \in \R\), we also have
  \begin{align*}
    \int_{\partial_{\mathrm{red}}\{u > 0\}} \varphi(x)d\cH^{n-1}(x) &= \int_{\R}\int_{\partial_{\mathrm{red}}\{u > 0\} \cap \{x_n = t\}} \psi(x')\eta(t)d\cH^{n-2}(x')dt \\
                                                    &= \int_{\R} \eta(t) \int_{\partial_{\mathrm{red}}\{v > 0\}} \psi(x')d\cH^{n-2}(x')dt \\
                                                    &= \int_{\partial_{\mathrm{red}}\{v > 0\}} \psi(x')d\cH^{n-2}.
  \end{align*}
  Thus, using that \(u\) is a weak solution we conclude that
  \[
    -\int_{\R^{n-1}} \nabla v\cdot\nabla\psi dx' =
    -\int_{\R^n} \nabla u\cdot\nabla\varphi dx =
    Q_0\int_{\partial_{\mathrm{red}}\{u > 0\}} \varphi d\cH^{n-1} =
    Q_0\int_{\partial_{\mathrm{red}}\{v > 0\}} \psi d\cH^{n-2}.
  \]
  For the converse, assume \(v\) is a weak solution in \(\R^{n-1}\). By
  \eqref{eq:7}, for any function \(\varphi \in C^{\infty}_{\mathrm{c}}(\R^n)\) we have
  \[
    \int_{\R^n} \nabla u(x)\cdot\nabla\varphi(x)dx = \int_{\R^{n-1}} \nabla v(x')\cdot\nabla\psi(x')dx',
  \]
  where \(\psi = \int_{\R} \varphi dx_n\). On the other hand,
  \begin{align*}
    \int_{\partial_{\mathrm{red}}\{u > 0\}} \varphi(x)d\cH^{n-1}(x) &= \int_{\R}\int_{\partial_{\mathrm{red}}\{u > 0\} \cap \{x_n = t\}} \varphi(x',t)d\cH^{n-2}(x')dt \\
                                                    &= \int_{\R}\int_{\partial_{\mathrm{red}}\{v > 0\}} \varphi(x',t)d\cH^{n-2}(x')dt \\
                                                    &= \int_{\partial_{\mathrm{red}}\{v > 0\}} \psi(x')d\cH^{n-2},
  \end{align*}
  and therefore
  \[
    -\int_{\R^n} \nabla u\cdot\nabla\varphi dx =
    -\int_{\R^{n-1}} \nabla v\cdot\nabla\psi dx' =
    Q_0\int_{\partial_{\mathrm{red}}\{v > 0\}} \psi d\cH^{n-2} =
    Q_0\int_{\partial_{\mathrm{red}}\{u > 0\}} \varphi d\cH^{n-1}.
  \]
\end{proof}
\begin{definition}
  We say that \(u:\R^n \to \R\) is a Lipschitz cone if \(u\) is a 1-homogeneous
  weak solution with Lipschitz free boundary, \(f \equiv 0\) and
  \(Q \equiv Q_0\) for some constant \(Q_0 > 0\). A Lipschitz cone is said to be
  trivial if it's free boundary is a hyperplane, that is, there exists a unit
  vector \(\nu \in \R^n\) such that \(u(x) = Q_0(x\cdot\nu)_+\).
\end{definition}
\begin{lemma}
  \label{lem:6}
  Let \(u\) be a Lipschitz cone in \(\R^n\) such that
  \(x_0 = e_n \in \partial\{u > 0\}\). Then, any blow-up sequence
  \[
    u_r(x) = \frac{1}{r}u(x_0 + rx)
  \]
  has a subsequence \(u_{r_k}\), \(r_k \rightarrow 0\) which converges locally uniformly
  to
  \[
    u_0(x_1,\ldots,x_n) = v(x_1,\ldots,x_{n-1}),
  \]
  with \(v\) a Lipschitz cone in \(\R^{n-1}\). Moreover, if \(x_0\) is a
  singular point for \(\partial\{u > 0\}\), then \(v\) is a non-trivial cone.
\end{lemma}
\begin{proof}
  We follow closely the proof of \cite[Lemma 5.4]{MR3353802}. Let \(Q_0 > 0\) be the
  boundary condition of \(u\). By Corollary \ref{cor:3} and Proposition \ref{prop:7},
  we can find a subsequence \(u_{r_k}\) with \(r_k \rightarrow 0\) converging locally uniformly
  to a 1-homogeneous weak solution \(u_0\) with the same constant boundary data
  \(Q_0\). Moreover, for any fixed \(t \in \R\), we can take \(r\) small enough so that
  \(1 + tr \ge 0\). Therefore, using that \(u\) is 1-homogeneous we see that
  \begin{align*}
    u_r(x) &= \frac{1}{r(1 + tr)}u((1 + tr)(x_0 + rx)) \\
           &= \frac{1}{1 + tr}u_r(tx_0 + (1 + tr)x),
  \end{align*}
  and letting \(r = r_k \rightarrow 0\) we obtain that
  \[
    u_0(x) = u_0(tx_0 + x), \quad\text{for all } t \in \R.
  \]
  Thus, \(u_0\) is constant in the \(x_0 = e_n\) direction. By Proposition
  \ref{prop:8} it follows that \(v(x') = u_0(x',x_n)\) is a weak weak solution
  in \(\R^{n-1}\) with \(f \equiv 0\), \(Q \equiv Q_0\) and Lipschitz free boundary.
  Moreover, \(v\) is 1-homogeneous since
  \[
    v(\lambda x') = u_0(\lambda x',x_{n+1}) = u_0(\lambda x',\lambda x_{n+1}) = \lambda u_0(x',x_{n+1}) = \lambda v(x')
  \]
  for any \(\lambda > 0\). Hence, \(v\) is a Lipschitz cone in \(\R^{n-1}\).

  For the final statement, assume \(v\) is a trivial cone, that is,
  \(v(x') = Q_0(x'\cdot\nu')_+\) for some unit vector \(\nu' \in \R^{n-1}\). Then
  \[
    u_0(x) = v(x') = Q_0(x'\cdot\nu')_+ = Q_0(x\cdot\nu)_+
  \]
  where \(\nu \coloneqq (\nu',0)\). Thus, \(u_0\) is a trivial cone in
  \(\R^n\) and therefore \(x_0\) is a regular point.
\end{proof}
\begin{lemma}
  \label{lem:7}
  Let \(v \geq 0\) be a non-negative function defined in \(B_1\) such that
  \(v\) is harmonic in \(\{v > 0\}\), \(0 \in \partial\{v > 0\} \cap B_1\) and
  \(\partial\{v > 0\} \cap B_1\) is a Lipschitz graph in the \(e_n\) direction with
  Lipschitz constant \(L\). Then there exists \(\delta = \delta(n,L) > 0\) such that
  \(v\) is monotone in the \(e_n\) direction in \(B_{\delta}\).
\end{lemma}
\begin{proof}
  See \cite[Lemma 11.12]{MR2145284}.
\end{proof}
\begin{proposition}
  \label{prop:9}
  All Lipschitz cones are trivial.
\end{proposition}
\begin{proof}
  We proceed by induction on the dimension \(n\). Assume \(n = 2\) and let
  \(u:\R^2 \to \R\) be a Lipschitz cone. Since \(\partial\{u > 0\}\) is Lipschitz,
  \(u\) satisfies the density estimate \eqref{eq:densityEst} and we can apply
  \cite[Theorem 5.1]{MR1620644} to obtain that
  \(\partial\{u > 0\} = \partial_{\mathrm{red}}\{u > 0\}\). Hence, \cite[Theorem 8.2]{MR618549} and
  \cite[Corollary 1.6]{carducci2025regularity} imply that \(\partial\{u > 0\}\) is a
  \(C^{\infty}\)-curve locally in \(\R^2\), and by the 1-homogeneity of \(u\) we conclude
  \(u\) is trivial.

  Assume now that the statement holds for dimension \(n - 1\) and let \(u\) be a
  Lipschitz cone in \(\R^n\) with boundary condition \(Q_0\). Since all Lipschitz
  cones in \(\R^{n-1}\) are trivial by induction hypothesis, Lemma \ref{lem:6}
  implies that \(\partial\{u > 0\} \setminus \{0\}\) consists only of regular points. Thus,
  \(\partial\{u > 0\}\) is smooth outside the origin by Corollary \ref{cor:2}. Moreover,
  without loss of generality we may assume that \(\partial\{u > 0\}\) is a Lipschitz graph
  in the \(e_n\) direction in \(B_1\) with Lipschitz constant \(L\). Hence, given any
  direction \(\sigma\) of the open cone
  \[
    \mathcal{C} \coloneqq \{\xi = (\xi',\xi_n) \in \R^n \mid \xi_n > L|\xi'|\},
  \]
  there exists \(\delta(\sigma)\) such that \(\partial_{\sigma}u \geq 0\) in
  \(B_{\delta(\sigma)}\) by Lemma \ref{lem:7}. Since \(u\) is 1-homogeneous,
  \(\partial_{\sigma}u\) is 0-homogeneous and therefore \(\partial_{\sigma}u \geq 0\) in
  \(\{u > 0\}\) for any \(\sigma \in \mathcal{C}\). Moreover, by continuity of the derivatives of
  \(u\) in \(\{u > 0\}\), the same conclusion must hold for
  \(\sigma \in \partial\mathcal{C}\). Thus, there exists a direction
  \(\tau \in \partial\mathcal{C}\), \(|\tau| = 1\) such that \(\tau\) is tangent to
  \(\partial\{u > 0\}\) at some point \(x_0 \in \partial\{u > 0\} \setminus \{0\}\) and
  \[
    \partial_{\tau}u \geq 0 \quad\text{in } \{u > 0\}.
  \]

  If \(\partial_{\tau}u = 0\) at some point in \(\{u > 0\}\), then
  \(\partial_{\tau}u \equiv 0\) by the maximum principle. Hence, \(u\) is constant in the
  \(\tau\) direction and by Proposition \ref{prop:8} we can reduce the problem to
  \(n - 1\) dimensions. Then by induction hypothesis \(u\) is trivial and we
  are done.

  Otherwise \(\partial_{\tau}u > 0\) in \(\{u > 0\}\). Since \(x_0 \neq 0\), the free boundary is
  smooth in a neighborhood \(B_{\rho}(x_0)\) of \(x_0\). Thus, Remark \ref{rem:3}
  implies that \(u\) is a classical solution in \(B_{\rho}(x_0)\) and
  \(u \in C^{\infty}(\overline{\{u > 0 \}} \cap B_{\rho}(x_0))\). In particular, if
  \(\nu\) denotes the inward normal vector to \(\partial\{u > 0\}\), then
  \(\partial_{\nu}u = Q_0\) on \(\partial\{u > 0\} \cap B_{\rho}(x_0)\). Therefore, the fact that
  \(\tau\) is tangent to \(\partial\{u > 0\}\) at \(x_0\) yields
  \[
    \partial_{\nu}\partial_{\tau}u(x_0) = 0.
  \]
  On the other hand, \(\partial_{\tau}u\) is a positive harmonic function in
  \(\{u > 0\} \cap B_{\rho}(x_0)\) (also smooth up to the boundary) and
  \(\partial_{\tau}u \geq 0\) on \(\partial\{u > 0\} \cap B_{\rho}(x_0)\) by continuity of
  \(\partial_{\tau}u\). Since \(\tau \in \partial\mathcal{C}\) is tangent to
  \(\partial\{u > 0\}\) at \(x_0\), we also have that \(\partial_{\tau}u(x_0) = 0\). Thus, we reach a
  contradiction because the Hopf lemma implies that
  \[
    \partial_{\nu}\partial_{\tau}u(x_0) > 0.
  \]
\end{proof}
With Proposition \ref{prop:9} we are now able to give the first proof of Theorem
\ref{thm:1}.
\begin{proof}[First proof of Theorem \ref{thm:1}]
  By Lemma \ref{lem:5} \(u\) is a weak solution in \(B_1\) in the sense of Definition
  \ref{defn:1} and its free boundary
  \(\partial\{u > 0\} \cap B_1 = \partial\Omega \cap B_1\) is Lipschitz. Let
  \(x_0 \in \partial\{u > 0\} \cap B_1\) be a free boundary point. By Corollary \ref{cor:3} and
  Proposition \ref{prop:7}, for any blow--up sequence
  \[
    u_{r_k}(x) \coloneqq \frac{1}{r_k}u(x_0 + r_kx), \quad r_k \to 0,
  \]
  there exists a 1-homogeneous weak solution \(u_0\) with \(f \equiv 0\) and
  \(Q \equiv Q(x_0)\) such that, up to subsequence, \(u_{r_k} \to u_0\) in the sense of
  Lemma \ref{lem:1}. In particular, \(u_0\) is a Lipschitz cone which must be trivial
  by Proposition \ref{prop:9}, that is, there exists a unit vector \(\nu \in \R^n\) such
  that
  \[
    u_0(x) = Q(x_0)(x\cdot\nu)_+.
  \]
  Then by Corollary \ref{cor:2} \(\partial\{u > 0\}\) is smooth in a neighborhood of
  \(x_0\). It follows that \(\partial\{u > 0\} = \partial\Omega \cap B_1\) is locally a
  \(C^{\infty}\)-surface in \(B_1\).
\end{proof}

\subsection{Proof using viscosity solutions}

In general, a weak solution \(u\) need not be a viscosity solution. However, if the
density of \(\{u = 0\}\) at the free boundary is positive, then we show that \(u\) is
a viscosity solution. This allows us to prove the improved regularity of Lipschitz
free boundaries using the known regularity results for viscosity solutions.
\begin{definition}
  \label{defn:3}
  A non-negative function \(u \in C(B_1)\) is a viscosity solution of
  \eqref{eq:inhom-alt-caff-free-bound-form} if
  \begin{enumerate}
  \item \(-\Delta u = f\) in \(\{u > 0\}\) in the viscosity sense.
  \item \(u\) satisfies the boundary condition \(|\nabla u| = \partial_{\nu}u = Q\) in the
    following sense: for any point \(x_0 \in \partial\{u > 0\} \cap B_1\) and any smooth function
    \(\varphi \in C^{\infty}(B_1)\),
    \begin{itemize}
    \item If \(\varphi\) touches \(u\) from below at \(x_0\), then
      \(|\nabla\varphi(x_0)| \le Q(x_0)\).
    \item If \(\varphi_+\) touches \(u\) from above at \(x_0\), then
      \(|\nabla\varphi(x_0)| \ge Q(x_0)\).
    \end{itemize}
  \end{enumerate}
\end{definition}
\begin{proposition}
  \label{prop:10}
  Let \(u\) be a weak solution of \eqref{eq:inhom-alt-caff-free-bound-form} in
  \(B_1\) such that
  \begin{equation}
    \label{eq:pos_lower_density}
    \liminf_{r \to 0} \frac{|\{u = 0\} \cap B_r(x)|}{|B_r|} > 0
  \end{equation}
  for every point \(x \in \partial\{u > 0\} \cap B_1\). Then \(u\) is a viscosity solution in
  \(B_1\).
\end{proposition}
\begin{proof}
  We follow the proof of \cite[Proposition 9.18]{MR4807210}. We need to check that
  \(u\) satisfies the boundary condition \(\partial_{\nu}u = Q\) in the viscosity sense. To
  this end, let \(x_0 \in \partial\{u > 0\} \cap B_1\) and let \(\varphi\) be a smooth function
  touching \(u\) from below at \(x_0\). Consider the blow-up sequences
  \[
    u_{r_k}(x) = \frac{1}{r_n}u(x_0 + r_kx), \qquad
    \varphi_{r_k}(x) = \frac{1}{r_k}\varphi(x_0 + r_kx).
  \]
  By \eqref{eq:pos_lower_density} we can apply Proposition \ref{prop:5} to obtain the
  existence a weak solution \(u_0 \in C^{0,1}(\R^n)\) with \(f \equiv 0\) and
  \(Q \equiv Q(x_0)\) such that \(u_{r_k} \to u_0\) locally uniformly in \(\R^n\). Moreover,
  \(u_0\) is harmonic in \(\{u_0 > 0\}\) and by Proposition \ref{prop:7} it is
  1-homogeneous.

  Notice that by the smoothness of \(\varphi\) we have that
  \(\varphi_{r_k} \to \varphi_0\) locally uniformly in \(\R^n\) with
  \(\varphi_0(x) = \nabla\varphi(x_0) \cdot x\). Without loss of generality we may assume that
  \(\nabla\varphi(x_0) = ae_n\) so that
  \[
    \varphi_0(x) = ax_n.
  \]
  Moreover, we can assume that \(a > 0\) since otherwise the inequality
  \(|\nabla\varphi(x_0)| \leq Q(x_0)\) holds trivially.

  Given the fact that \(u \geq \varphi\) near \(x_0\), we see that
  \(u_0 \geq \varphi_0\) and therefore \(u_0 > 0\) on the set \(\{x_n > 0\}\). Thus,
  \(u_0\) is a 1-homogeneous harmonic function on \(\{u_0 > 0\} \supset \{x_n > 0\}\) and
  we deduce (see \cite[Lemma 9.16]{MR4807210}) that
  \[
    \text{either}\quad u_0(x) = \alpha x_n^+ \quad\text{or}\quad u_0(x) = \alpha x_n^+ + \beta x_n^-.
  \]
  The second case can be discarded since it contradicts
  \eqref{eq:pos_lower_density}. Hence,
  \[
    u_0(x) = \alpha x_n^+
  \]
  and Remark \ref{rem:3} implies that \(\alpha = Q(x_0)\). Now, the inequality
  \(u_0 \geq \varphi_0\) implies that
  \[
    |\nabla\varphi(x_0)| = a \leq Q(x_0).
  \]

  Suppose now that \(\varphi\) touches \(u\) from above at \(x_0\). We repeat the same
  blow-up procedure and assume \(\varphi_0\) is of the form \(\varphi_0(x) = ax_n\) with
  \(a = |\nabla\varphi(x_0)|\). Again, since \(u \leq \varphi\) near \(x_0\), we obtain that
  \(u_0 \leq \varphi_0\). Now, the non-degeneracy of \(u\) implies that
  \(u_0 \not\equiv 0\) and therefore \(a > 0\). Moreover, the inequality
  \(u_0 \leq \varphi_0\) also implies \(\{u_0 > 0\} \subset \{x_n > 0\}\). Hence, by the
  1-homogeneity of \(u_0\) and \cite[Lemma 9.15]{MR4807210} we must have
  \(\{u_0 > 0\} = \{x_n > 0\}\) and
  \[
    u_0(x) = \alpha x_n^+.
  \]
  Thus, by Remark \ref{rem:3} we get that \(\alpha = Q(x_0)\) and the inequality
  \(u_0 \leq \varphi_0\) implies
  \[
    |\nabla\varphi(x_0)| = a \geq Q(x_0).
  \]
\end{proof}
\begin{proof}[Second proof of Theorem \ref{thm:1}]
  By Lemma \ref{lem:5} \(u\) is a weak solution in \(B_1\) in the sense of Definition
  \ref{defn:1} and its free boundary
  \(\partial\{u > 0\} \cap B_1 = \partial\Omega \cap B_1\) is Lipschitz. In particular,
  \(u\) satisfies \eqref{eq:pos_lower_density}. Thus, \(u\) is a viscosity solution
  in \(B_1\) by Proposition \ref{prop:10} and we can apply \cite[Theorem
  1.2]{MR2813524} to obtain that \(\partial\{u > 0\} \cap B_1\) is locally a
  \(C^{1,\alpha}\)-surface. Finally, \cite[Corollary 1.6]{carducci2025regularity} allows
  us to improve this regularity and conclude that \(\partial\{u > 0\} \cap B_1\) is locally a
  \(C^{\infty}\)-surface as claimed.
\end{proof}

\section{Applications of Theorem \ref{thm:1}}
\label{sec:main-results}

\subsection{Serrin's overdetermined problem}

Recall that for a bounded Lipschitz domain \(\Omega \subset \R^n\), a function
\(u \in H^1_0(\Omega)\) is said to be a weak solution of Serrin's problem
\[
  \left\{
  \begin{aligned}
    -\Delta u &= 1 \quad&&\text{in } \Omega \\
    u &= 0 \quad&&\text{on } \partial \Omega \\
    \partial_{\nu}u &= c > 0 \quad&&\text{on } \partial \Omega,
  \end{aligned}
  \right.
\]
if it satisfies the equation
\begin{equation}
  \label{eq:weak-serrin-formulation}
  -\int_\Omega \nabla u\cdot\nabla\eta = c\int_{\partial \Omega} \eta d\cH^{n-1} - \int_\Omega \eta, \qquad \forall\eta \in C^{\infty}_{\mathrm{c}}(\R^n).
\end{equation}
\begin{proof}[Proof of Theorem \ref{thm:2}]
  Without loss of generality we may assume \(0 \in \partial \Omega\) and
  \(\Omega \subset\subset B_1\). Note that \eqref{eq:weak-serrin-formulation} implies
  \[
    \int_\Omega \nabla u\cdot\nabla\eta = \int_\Omega \eta \geq 0, \qquad \forall\eta \in C^{\infty}_{\mathrm{c}}(\Omega),\, \eta \geq 0.
  \]
  Thus, \(u\) is a weakly superharmonic function in \(\Omega\) with zero boundary values.
  By the maximum principle, \(u > 0\) in \(\Omega\). After extending \(u\) by zero in
  \(B_1 \setminus \Omega\), we can take \(\eta \in C^{\infty}_{\mathrm{c}}(B_1)\) in
  \eqref{eq:weak-serrin-formulation} to obtain
  \[
    -\int_{B_1} \nabla u\cdot\nabla\eta = c\int_{\partial \Omega \cap B_1} \eta d\cH^{n-1} - \int_{\Omega \cap B_1} \eta, \qquad \forall\eta \in C^{\infty}_{\mathrm{c}}(B_1).
  \]
  Hence, by Theorem \ref{thm:1} \(\partial \Omega \cap B_1\) is locally smooth in
  \(B_1\), so \(\partial \Omega \cap B_{1/2} = \partial \Omega\) is smooth. Now, by Schauder estimates
  \(u \in C^{\infty}(\overline{\Omega})\) and therefore \(u\) solves \eqref{eq:serrin-prob} in
  the classical sense. Thus, \cite{MR333220} implies \(\Omega = B_R\).
\end{proof}

\subsection{Poisson kernel regularity}

Recall that if \(\Omega \subset \R^n\) is a bounded Lipschitz domain and \(G_x\) denotes the
Green function for \(\Omega\) with pole at \(x \in \Omega\), then \(G_x\) is a positive harmonic
function in \(\Omega \setminus \{x\}\),
\(G_x \in C(\overline{\Omega} \setminus \{x\})\) and \(G_x = 0\) on
\(\partial\Omega\). Moreover, the Poisson kernel \(P_x = \partial_{\nu}G_x\) is a positive function
defined \(\cH^{n-1}\)-a.e. on \(\partial\Omega\) and satisfies the equation
\[
  -\int_\Omega \nabla G_x\cdot\nabla\eta = -\eta(x) + \int_{\partial\Omega} P_x\eta d\cH^{n-1}, \qquad \forall\eta \in C^{\infty}_{\mathrm{c}}(\R^n).
\]
\begin{proof}[Proof of Theorem \ref{thm:3}]
  Let \(G_x\) be the Green function for \(\Omega\) with pole at \(x\). Also, let
  \(y \in \partial \Omega\) be any boundary point and choose \(r\) sufficiently small so that
  \(x \not\in B_r(y)\). After a translation and a rescaling we may assume \(y = 0\) and
  \(r = 1\). Then, by the discussion above
  \[
    -\int_{B_1} \nabla G_x\cdot\nabla\eta = \int_{\partial \Omega \cap B_1} P_x\eta d\cH^{n-1}, \qquad \forall\eta \in C^{\infty}_{\mathrm{c}}(B_1),
  \]
  where \(P_x\) is smooth by assumption and positive on \(\partial\Omega\). Thus, extending
  \(G_x\) by zero in \(B_1 \setminus \Omega\) we can apply Theorem \ref{thm:1} to deduce that
  \(\partial\Omega\) is a smooth surface in, say, \(B_{1/2}\). Since
  \(y \in \partial\Omega\) is arbitrary, we conclude that \(\partial\Omega\) is a smooth surface.
\end{proof}

\printbibliography
\end{document}